\documentclass[10pt]{article}

\usepackage{amsmath,amssymb,amsfonts,amsthm,epic}
\usepackage[vcentermath]{youngtab}
\usepackage{graphicx}

\theoremstyle{plain}
\newtheorem{thm}[subsubsection]{Theorem}
\newtheorem{lem}[subsubsection]{Lemma}
\newtheorem{prop}[subsubsection]{Proposition}
\newtheorem{cor}[subsubsection]{Corollary}
\newtheorem{ex}[subsubsection]{Example}

\theoremstyle{definition}
\newtheorem{rem}[subsubsection]{Remark}
\newtheorem{defn}[subsubsection]{Definition}

\newcommand\brem{\begin{rem}\begin{sffamily}\begin{upshape}}
\newcommand\erem{\end{upshape}\end{sffamily}\end{rem}}
\newcommand\bdefn{\begin{defn}\begin{rm}}
\newcommand\edefn{\end{rm}\hfill$\Box$\end{defn}}
\newcommand\bex{\begin{ex}\begin{rm}}
\newcommand\eex{\end{rm}\hfill$\Box$\end{ex}}

\newenvironment{Proof}{%
\par\noindent{\scshape Proof:}\begin{rm}}{\hfill$\Box$\end{rm}\newline}

\newcommand{\es}{\emptyset}
\newcommand{\bN}{\mathbb{N}}

\newcommand{\ga}{\alpha}
\newcommand{\gb}{\beta}
\newcommand{\gc}{\gamma}
\newcommand{\gi}{\iota}

\newcommand{\gt}{\theta}

\newcommand{\disjunion}{\,\,\dot{\cup}\,\,}

\newcommand{\NN}{\bN^2}

\newcommand{\gB}{\ol{\gb}\times\gb}
\newcommand{\BBZ}{\gB}

\newcommand{\pr}{R}
\newcommand{\qr}{S}

\newcommand{\leqT}{\leq}
\newcommand{\gtT}{>}
\newcommand{\ltT}{<}

\newcommand{\ol}{\overline}

\newcommand{\la}[1]{\stackrel{#1}{\leftarrow}}
\newcommand{\ra}[1]{\stackrel{#1}{\rightarrow}}

\newcommand{\abc}{_{\ga,\gb}^\gc}

\newcommand\field{\mathfrak{k}}

\newcommand\id{I(d)}

\newcommand\modV{(V)}
\newcommand\misd{\mathfrak{M}_d\modV}
\newcommand\miso{\misd}

\newcommand\affinev{\mathbb{A}\modv}

\newcommand\ortho{\ensuremath{\mathfrak{O}}}

\newcommand\dortho{\ensuremath{\mathfrak{A}}}

\newcommand\pos{\mathfrak{N}(v)}
\newcommand\roots{\mathfrak{R}(v)}
\newcommand\posb{\mathfrak{N}(\gb)}
\newcommand\rootsb{\mathfrak{R}(\gb)}

\newcommand\rootsbstar{\mathfrak{R}(\gb^*)}

\newcommand\rootsvstar{\mathfrak{R}(v^*)}

\newcommand\modv{}

\newcommand\st{\,|\,}

\newcommand\posv{\ortho\pos\modv}

\newcommand\rootsv{\ortho\roots\modv}
\newcommand\andposv{\pos\modv}
\newcommand\androotsv{\roots\modv}
\newcommand\dposv{\dortho\pos\modv}

\newcommand\drootsv{\dortho\roots\modv}

\newcommand\orootsb{\ortho\rootsb}

\newcommand\drootsb{\dortho\rootsb\modv}

\newcommand\diag{\mathfrak{d}}
\newcommand\diagv{\diag^{v}}
\newcommand\diagb{\diag^{\gb}}

\newcommand\hash{\#}

\newcommand{\mon}{\mathfrak{S}}

\newcommand\path{\Lambda}

\newcommand\init{\textup{in}_{\torder}}
\newcommand\torder{\vartriangleright}

\newcommand\card\#
\numberwithin{equation}{subsection}
\numberwithin{figure}{subsection}

\title{Initial ideals of tangent cones to Richardson varieties in the Orthogonal Grassmannian via a Orthogonal-Bounded-RSK-Correspondence}
\date{}
\author {Shyamashree Upadhyay}

\begin{document}
\maketitle

\begin{abstract}
A Richardson variety $X_\ga^\gc$ in the Orthogonal Grassmannian is defined
to be the intersection of a Schubert variety $X^\gc$ in the Orthogonal Grassmannian and
a opposite Schubert variety $X_\ga$ therein.  We give an explicit description of the initial ideal (with respect to certain conveniently chosen term order) for the ideal of the tangent cone at any $T$-fixed point of
$X_\ga^\gc$, thus generalizing a result of Raghavan-Upadhyay
\cite{Ra-Up2}. Our proof is based on a generalization of the Robinson-Schensted-Knuth (RSK)
correspondence, which we call the Orthogonal bounded RSK (OBRSK). The OBRSK correspondence will give a degree-preserving bijection
between a set of monomials defined by the initial ideal of the ideal of the tangent cone (as mentioned above) and a
`standard monomial basis'. A similar work for Richardson varieties in the ordinary Grassmannian was done by Kreiman in \cite{Kr-bkrs}.  
\end{abstract}

\tableofcontents

\section{Introduction}

The Orthogonal Grassmannian is as defined in \S 1.1 of \cite{sthesis}. A Richardson variety $X_\ga^\gc$ in the Orthogonal
Grassmannian\footnote{Richardson varieties in the ordinary Grassmannian are
also studied by Stanley in \cite{St}, where these varieties are
called {\it skew Schubert varieties}.  Discussion of these
varieties in the ordinary Grassmannian also appears in \cite{Ho-Pe}.} is defined to be the
intersection of a Schubert variety $X^\gc$ in the Orthogonal Grassmannian with a opposite Schubert
variety $X_\ga$ therein.  In particular, Schubert and opposite Schubert
varieties are special cases of Richardson varieties.  In this paper, we provide an explicit description of the initial ideal (with respect to certain conveniently chosen term order) for the ideal of the tangent cone at any $T$-fixed point $e_\gb$ of
$X_\ga^\gc$. It should be noted that the local properties of Schubert varieties at
$T$-fixed points determine their local properties at all other
points, because of the $B$-action; but this does not extend to
Richardson varieties, since Richardson varieties only have a
$T$-action.

In Kodiyalam-Raghavan \cite{Ko-Ra} and Kreiman-Lakshmibai
\cite{Kr-La}, an explicit Gr\"obner basis for the ideal of the
tangent cone of the Schubert variety $X^\gc$ (in the ordinary Grassmannian) at any torus fixed point $e_\gb$ is
obtained. In Raghavan-Upadhyay \cite{Ra-Up2}, an explicit description of the initial ideal (with respect to certain conveniently chosen term orders) for the ideal of the tangent cone at any $T$-fixed point of a Schubert variety in the Orthogonal Grassmannian has been obtained. In this paper, we generalize the result of \cite{Ra-Up2} to the case of Richardson varieties in the Orthogonal Grassmannian.

Sturmfels \cite{Stu} and Herzog-Trung \cite{He-Tr} proved results
on a class of determinantal varieties which are equivalent to the
results of \cite{Ko-Ra}, \cite{Kr-La}, and \cite{Kr-bkrs} for the case
of Schubert varieties (in the ordinary Grassmannian) at the $T$-fixed point $e_{id}$. The key to
their proofs was to use a version of the Robinson-Schensted-Knuth
correspondence (which we shall call the `ordinary' RSK)
in order to establish a degree-preserving bijection
between a set of monomials defined by an initial ideal and a
`standard monomial basis'. The difficulty in generalizing this method of proof to the case of
Schubert varieties (in the ordinary Grassmannian) at an arbitrary $T$-fixed point $e_\gb$ lies in
generalizing this bijection. All three of \cite{Ko-Ra},
\cite{Kr-La}, and \cite{Kr-bkrs} obtain generalizations of this
bijection; the generalization in \cite{Kr-bkrs} is slightly more general, since it
applies to Richardson varieties, and not just to Schubert varieties. These three generalizations, when restricted to Schubert varieties in the ordinary Grassmannian, are in fact the same bijection\footnote{This supports the
conviction of the authors in \cite{Ko-Ra} that this bijection is
natural and that it is in some sense the only natural bijection
satisfying the required geometric conditions.}, although this is
not immediately apparent. In \cite{Ko-Ra} and \cite{Kr-La}, this `generalized bijection' is not viewed as a generalization of the `ordinary' RSK correspondence. It is only in the work of Kreiman in \cite{Kr-bkrs}, where this `generalized bijection' has been viewed as a generalization of the `ordinary' RSK correspondence, which he calls the Bounded-RSK correspondence. Although the formulations of the
bijections in \cite{Ko-Ra} and \cite{Kr-La} are similar to
eachother, the formulation of the bijection in \cite{Kr-bkrs} is in terms of different combinatorial
indexing sets. The relationship between the formulation in \cite{Kr-bkrs} and the
formulations in \cite{Ko-Ra} and \cite{Kr-La} is analogous to the
relationship between the Robinson-Schensted correspondence and
Viennot's version of the Robinson-Schensted correspondence \cite{Sa,Vi}.

Results analogous to those of \cite{Ko-Ra} and \cite{Kr-La} have
been obtained for the symplectic and orthogonal Grassmannians
(see \cite{Gh-Ra}, \cite{Ra-Up1}, \cite{Ra-Up2}). Given any torus fixed point in a Schubert variety in the Orthogonal Grassmannian, it is known (see, for instance \cite{Ra-Up2} or \cite{sthesis}) that the ideal of the tangent cone at this torus fixed point is generated by certain special kind of pfaffians. In the case when the Schubert variety is of a special kind and, the torus fixed point corresponds to the `identity coset', and the pfaffian generators of the ideal of the tangent cone are of a fixed size, Herzog and Trung provide a Gr\"obner basis for the ideal of the tangent cone in their paper \cite{He-Tr}. In the paper \cite{He-Tr}, Herzog and Trung use a version of the Robinson-Schensted-Knuth
correspondence (which we shall call the `ordinary' RSK)
in order to establish a degree-preserving bijection
between a set of monomials defined by an initial ideal and a
`standard monomial basis'. In \cite{Ra-Up2}, Raghavan and Upadhyay generalize the results of Herzog and Trung as in \cite{He-Tr} to ideals of tangent cones at any torus fixed point in any Schubert variety in the Orthogonal Grassmannian. In fact, Raghavan and Upadhyay give an explicit computation of the initial ideal (with respect to certain conveniently chosen
term orders) of the ideal of the tangent cone at any torus fixed point of any Schubert variety in the Orthogonal Grassmannian. But the computation in \cite{Ra-Up2} is done in the same spirit as in \cite{Ko-Ra} (for the ordinary Grassmannian) and \cite{Gh-Ra} (for the symplectic Grassmannian). The work done in \cite{Ra-Up2} does not involve any version of the RSK correspondence, unlike by Herzog and Trung in \cite{He-Tr}. The work done in \cite{Ra-Up2} relies on a degree-preserving bijection between a set of monomials defined by an initial ideal and a
`standard monomial basis', and this bijection is proved by Raghavan and Upadhyay in \cite{Ra-Up1}. It is mentioned in \cite{Ra-Up1} that it will be nice if the bijection proved therein can be viewed as a kind of `Bounded-RSK' correspondence, as done by Kreiman in \cite{Kr-bkrs} for the case of Richardson varieties in the ordinary Grassmannian. This paper fulfills the expectation made in \cite{Ra-Up1} of being able to view the bijection there as a generalized-bounded-RSK correspondence, which we call here the Orthogonal-bounded-RSK correspondence (OBRSK, for short). In fact, it is also mentioned by Kreiman in his paper \cite{Kr-bkrs} that he believes that it is possible to adapt the methods of \cite{Kr-bkrs} to Richardson varieties in the Symplectic and the Orthogonal Grassmannian as well. This paper also supports the above mentioned conviction of Kreiman made in his paper \cite{Kr-bkrs}. 

The OBRSK correspondence (as defined in this paper) is not a special case of the bounded-RSK correspondence as in \cite{Kr-bkrs}, however its basics rely upon the frame of the bounded-RSK correspondence. In fact, the OBRSK gives a bijective correspondence between certain special kind of pairs of multisets and certain special kind of bitableaux, unlike in the case of the bounded-RSK where the bijective correspondence was between certain special kind of multisets (not `pairs of multisets') and certain special kind of bitableaux. It will be nice if one can answer the following question:--- What can be an interpretation (in terms of representation theory of groups) of the fact that the bijection given in this paper is a generalized version of the RSK correspondence? In more details: It is proved in this paper that a set of certain special kind of bitableaux forms a basis for the coordinate ring of the tangent cone to a Richardson variety in the Orthogonal Grassmannian at any given torus fixed point of it. Now we can ask the following question:--- Does the above-mentioned set of special kind of bitableaux form a basis for modules of any group? If yes, then for what group? But before one asks such questions for the bijection given in this paper, the same questions need to be answered for the bijection given in the paper of Kreiman\cite{Kr-bkrs} in the case of Richardson varieties in the ordinary Grassmannian. And even before that, one needs to answer the question that what was the significance of the use of the RSK-correspondence in the work of Sturmfels (\cite{Stu}) and in the work of Herzog-Trung (\cite{He-Tr}). 
\subsection{Important note}
In this paper, we will be using lots of results, terminology and notation from \cite{sthesis} as well as \cite{Kr-bkrs}.
\subsection{Acknowledgements}
First of all I would like to thank Victor Kreiman whose paper (\cite{Kr-bkrs}) has motivated me to write this paper. I would also like to thank K. N. Raghavan, A. Conca and Sudhir. R. Ghorpade for valuable discussions, suggestions and corrections that had helped me in writing this paper.
\vspace{1em}
\section{The Orthogonal Grassmannian and Richardson varieties in it}\label{s.OGRV}
Fix an algebraically closed field $\field$ of characteristic not
equal to~$2$. Fix a natural number $d$, a vector space $V$ of dimension $2d$ over $\field$ and a non-degenerate symmetric bilinear form~$\langle\ ,\ \rangle$ on~$V$. For $k$ an integer such that $1\leq k\leq 2d$,  
set $k^*:=2d+1-k$.   Fix a basis 
$e_1,~\ldots,~e_{2d}$ of $V$ such that
\[
\langle e_i, e_k \rangle = \left\{ \begin{array}{rl}
    1 & \mbox{ if $i=k^*$ }\\ 
    0 & \mbox{ otherwise}\\ \end{array}\right.
 \]
Denote by $SO(V)$ the group of linear automorphisms of $V$ that preserve the bilinear 
form $\langle\ ,\ \rangle$, and also the volume form. A linear subspace of $V$ is said to be \textit{isotropic} if the bilinear form $\langle\ ,\ \rangle$ vanishes identically on it. Denote by $\miso'$ the closed sub-variety
of the Grassmannian of $d$-dimensional subspaces consisting of the points corresponding to maximal isotropic subspaces. The action of $SO(V)$ on $V$ induces an
action on $\miso'$ . There are two orbits for this action. These orbits are
isomorphic: acting by a linear automorphism that preserves the form but
not the volume form gives an isomorphism. We denote by $\miso$ the orbit
of the span of $e_1,\ldots,e_d$ and call it the \textit{(even) orthogonal Grassmannian}. One can define the Orthogonal Grassmannian in the case when the dimension of $V$ is not necessarily even, see \S 1.1 of \cite{sthesis} for instance. But it is enough to consider the case when the dimension of $V$ is even: this is proved in \S 1.3 of \cite{sthesis}. Therefore, now onwards we call the (even) orthogonal Grassmannian $\miso$ (as defined above for a $2d$ dimensional vector space $V$) to be the \textit{Orthogonal Grassmannian}. Let $\miso\subseteq G_{d}(V)\hookrightarrow \mathbb{P}(\wedge^d V)$ be the Pl\"ucker embedding (where $G_d(V)$ denotes the
Grassmannian of all $d$-dimensional subspaces of $V$). Thus $\miso$ is a closed subvariety of the projective variety $G_d(V)$, and hence $\miso$ inherits the structure of a projective variety.

We take $B$ (resp. $B^-$) to be the subgroup of $SO(V)$ consisting of those elements that are upper triangular (resp. lower triangular) 
with respect to the basis $e_1 ,\ldots, e_{2d}$, and the subgroup $T$ of $SO(V)$ consisting of those elements
that are diagonal with respect to $e_1,\ldots,e_{2d}$. It can be easily checked that $T$ is a maximal torus of $SO(V)$; $B$ and $B^-$ are Borel subgroups of $SO(V)$ which contain $T$. The group $SO(V)$ acts transitively on $\miso$, the $T$-fixed points of $\miso$ under this action are easily seen to be of the form $\langle e_{i_1},\ldots,e_{i_d}\rangle$ for $\{i_1,\ldots,i_d\}$ in $I(d)$, where $I(d)$ is the set of subsets of $\{1,\ldots,2d\}$ of cardinality $d$ satisfying the following two conditions:---\\
\begin{itemize}
\item for each $k$, $1\leq k\leq 2d$, the subset contains exactly one of $k$, $k^*$,
and
\item the number of elements in the subset that exceed~$d$ is even.
\end{itemize}
We write $I(d,2d)$ for the set of all $d$-element subsets of $\{1,\ldots,2d\}$.
There is a natural partial order on $I(d,2d)$ and so
also on $\id$:  $v=(v_1<\ldots<v_d)\leq w=(w_1<\ldots<w_d)$ if and only
if $v_1\leq w_1$, \ldots, $v_d\leq w_d$. For $\mu=\{\mu_1,\ldots,\mu_d\}\in I(d,2d), \mu_1<\cdots<\mu_d$, define the \textbf{complement} of $\mu$ as $\{1,\ldots,2d\}\setminus\mu$ and denote it by $\ol{\mu}$. 

The $B$-orbits (as well as $B^-$-orbits) of $\miso$ are naturally indexed by its $T$ -fixed points: each
$B$-orbit (as well as $B^-$-orbit) contains one and only one such point. Let $\ga\in\id$ be arbitrary and let $e_\ga$ denote the corresponding $T$-fixed point of $\miso$. The Zariski closure of the $B$ (resp. $B^-$) orbit through $e_\ga$, with canonical reduced
scheme structure, is called a \textbf{Schubert variety} (resp.
\textbf{opposite Schubert variety}), and denoted by $X^\ga$ (resp.
$X_\ga$).  For $\ga,\gc\in\id$, the scheme-theoretic
intersection $X_\ga^\gc=X_\ga\cap X^\gc$ is called a
\textbf{Richardson variety}. Each $B$-orbit (as well as $B^-$-orbit) being
irreducible and open in its closure, it follows that $B$-orbit closures (resp. $B^-$-orbit closures) are indexed
by the $B$-orbits (resp. $B^-$-orbits). Thus the set $\id$ becomes an indexing set for Schubert varieties in $\miso$, and the set consisting of all pairs of elements of $\id$ becomes an indexing set for Richardson varieties in $\miso$. It can be shown that $X_\ga^\gc$ is
nonempty if and only if $\ga\leq \gc$; that for $\gb\in\id$, $e_\gb\in X_\ga^\gc$
if and only if $\ga\leq\gb\leq\gc$; and that if $X_\ga^\gc$ is
nonempty, it is reduced and irreducible (see \cite{Br,Kr-La3,La-Go,Ri}). 
 
\section{Statement of the problem and the strategy of the proof}\label{s.state-problem}
In this section, we will first make an initial statement of the problem tackled in this paper, and then in different subsections of this section, we will develop necessary concepts, terminology and notation to make a statement of the main theorem (This will happen in the last subsection of this section, the main theorem being Theorem~\ref{t.main}), which will solve the problem tackled in this paper. Also in the last subsection, we will give a strategy of the proof. 
\subsection{Initial statement of the problem}\label{ss.In-state-problem}
The problem that is tackled in this paper is this: given
a $T$-fixed point on a Richardson variety in $\miso$, compute the initial ideal, with
respect to some convenient term order, of the ideal of functions vanishing on
the tangent cone to the Richardson variety at the given $T$-fixed point. The term order
is specified in \ref{ss.termorder}, and the answer is given in Theorem ~\ref{t.main}. 

For the rest of this paper, $\ga,\gb,\gc$ are arbitrarily fixed elements of $\id$ such that $\ga\leq\gb\leq\gc$. So, the problem tackled in this paper can be restated as follows: Given the Richardson variety $X_\ga^\gc$ in $\miso$ and the $T$-fixed point $e_\gb$ in it, find the initial ideal of the ideal of functions vanishing on the tangent cone at $e_\gb$ to $X_\ga^\gc$, with respect to some conveniently chosen term order. The tangent cone being a subvariety of the tangent
space at $e_\gb$ to $\miso$, we first choose a convenient set of coordinates for the
tangent space. But for that we need to fix some notation.

\subsection{Basic notation}\label{ss.basicnot}
For this subsection, let us fix an arbitrary element $v$ of $I(d,2d)$. We will be dealing extensively with ordered pairs $(r,c)$,
$1\leq r,c\leq 2d$,  such that $r$ is not and $c$ is an entry of~$v$.
Let $\androotsv$ denote the set of all such ordered pairs, and set
\\[1mm]
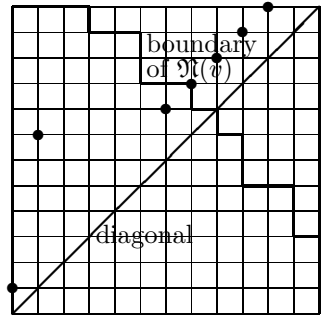
\begin{minipage}{6cm}
\begin{align*}
  \andposv &:= \left\{(r,c)\in\androotsv\st r>c\right\}\\
  \rootsv &:=  \left\{(r,c)\in\androotsv\st r<c^*\right\}\\
  \posv &:= \left\{(r,c)\in\androotsv\st r>c, r<c^*\right\}\\
  &=\rootsv\cap\andposv\\
  \diagv &:= \left\{(r,c)\in\androotsv\st r=c^*\right\}\\
\end{align*}\vfill
\end{minipage}
\hfill \begin{minipage}{6cm}
\setlength{\unitlength}{.34cm}
\begin{picture}(12,12)(-3,0)
\multiput(0,0)(12,0){2}{\line(0,1){12}}
\multiput(0,0)(0,12){2}{\line(1,0){12}}
\linethickness{.05mm}
\multiput(0,0)(1,0){13}{\line(0,1){12}}
\multiput(0,0)(0,1){13}{\line(1,0){12}}
\thicklines
\put(0,0){\line(1,1){12}}
\put(3.25,2.75){diagonal}
\thicklines
\put(5.25,10.25){boundary}\put(5.25,9.25){of $\andposv$}
\put(0,12){\line(1,0){3}}
\put(3,12){\line(0,-1){1}}
\put(3,11){\line(1,0){2}}
\put(5,11){\line(0,-1){2}}
\put(5,9){\line(1,0){2}}
\put(7,9){\line(0,-1){1}}
\put(7,8){\line(1,0){1}}
\put(8,8){\line(0,-1){1}}
\put(8,7){\line(1,0){1}}
\put(9,7){\line(0,-1){2}}
\put(9,5){\line(1,0){2}}
\put(11,5){\line(0,-1){2}}
\put(11,3){\line(1,0){1}}
\put(12,3){\line(0,-1){3}}
\put(0,1){\circle*{.4}}
\put(1,7){\circle*{.4}}
\put(6,8){\circle*{.4}}
\put(7,9){\circle*{.4}}
\put(8,10){\circle*{.4}}
\put(9,11){\circle*{.4}}
\put(10,12){\circle*{.4}}
\end{picture}
\end{minipage}
\\

The picture shows a drawing of $\androotsv$.   We think of $r$ and $c$
in $(r,c)$ as row index and column index respectively.    The columns
are indexed from left to right 
by the entries of~$v$ in ascending order,  the rows from top to bottom by
the entries of $\{1,\ldots,2d\}\setminus v$ in ascending order.
The points of $\diagv$ are those on the diagonal, 
the points of $\rootsv$ are those that are (strictly) above the diagonal, and
the points of $\andposv$ are those that are to the South-West of the
poly-line captioned ``boundary of $\andposv$''---we draw the
boundary so that points on the boundary belong to $\andposv$.
The reader can readily verify that  $d=13$ and $v=(1,2,3,4,6,7,10,11,13,15,18,19,22)$ for the particular picture drawn.   The points of $\rootsv$ indicated by solid circles form an \em{extended} $v$-\em{chain} (see the figure above), the definition of an extended $v$-chain is given later in \S~\ref{ss.extd-v-chains}.
 
We will be considering {\em monomials} (also called 
{\em multisets}) in some of these sets.
A {\em monomial},  as usual,  is a subset with each member being allowed
a multiplicity (taking values in the non-negative integers).  The 
{\em degree\/} of a monomial has also the usual sense: 
it is the sum of the multiplicities in the
monomial over all elements of the set.  The {\em intersection\/} 
of a monomial in a set with a subset of the set has also the natural meaning:
it is a monomial in the subset,  the multiplicities being those in the
original monomial. We will refer to $\diagv$ as the {\em diagonal}.

\begin{center}
Moreover, let $\drootsv:=\{(r,c)\in\androotsv\st r>c^*\}$\\
and  $\dposv:=\{(r,c)\in\androotsv\st r>c, r>c^*\}$
\end{center}
In other words, $\drootsv$ denotes the part of the grid (as in the picture above) that lies strictly below the {\em diagonal} and $\dposv$ denotes the intersection of $\drootsv$ with $\andposv$.

Given any two multisets $A$ and $B$ consisting of elements of $\androotsv$, let $set(A)$ and $set(B)$ denote the underlying sets of $A$ and $B$ respectively. We say that $B\subseteq A (as\ multisets)$ if $set(B)\subseteq set(A)$ and the multiplicity with which every element occurs in the multiset $B$ is less than or equal to the multiplicity with which the same element occurs in the multiset $A$. Given two multisets $A$ and $B$ consisting of elements of $\androotsv$ such that $B\subseteq A (as\ multisets)$, we can define a multiset called the \textit{`multiset minus' of $B$ from $A$} (denoted by $A\setminus_m B$) as follows: Take any element $x$ of $set(B)$. If the multiplicity with which $x$ occurs in $A$ is $m_x(A)$ and the multiplicity with which $x$ occurs in $B$ is $m_x(B)$, then the multiplicity with which $x$ occurs in the multiset $A\setminus_m B$ is $m_x(A)-m_x(B)$. And any element in $set(A)\setminus set(B)$ occurs in the multiset $A\setminus_m B$ with the same multiplicity with which it occurs in $A$. This finishes the description of $A\setminus_m B$.

\brem\label{r.idd}
Note that in this subsection, $v$ was any element of $I(d,2d)$, $v$ was not necessarily in $\id$. In particular, all the above basic notation will hold true if we take $v\in\id$ as well. 
\erem

\subsection{The tangent space to $\miso$ at $e_\gb$}\label{ss.tgtspace}
Let $\miso\subseteq G_{d}(V)\hookrightarrow \mathbb{P}(\wedge^d V)$ be the Pl\"ucker embedding (where $G_d(V)$ denotes the
Grassmannian of all $d$-dimensional subspaces of $V$). For $\theta$ in $I(d, 2d)$, let $p_\theta$ denote the corresponding Pl\"ucker coordinate. Consider the affine patch $\affinev$ of $\mathbb{P}(\wedge^d V)$ given by $p_\gb\neq 0$, where $\gb$ is the element of $\id$ which was fixed at the beginning of this section. The affine patch $\affinev^\gb:=\miso\cap\affinev$ of the orthogonal
Grassmannian $\miso$ is an affine space whose coordinate ring can be taken
to be the polynomial ring in variables of the form $X_{(r,c)}$ with $(r, c)\in\orootsb$.
Taking $d=5$ and $\gb=(1,3,4,6,9)$ for example, a general element of $\affinev^\gb$ has
a basis consisting of column vectors of a matrix of the following form:
\begin{equation}\label{eq.matrix}\left(\begin{array}{ccccc}
1 & 0 & 0 & 0 & 0 \\
X_{21} & X_{23} & X_{24} & X_{26} & 0 \\
0 & 1 & 0 & 0 & 0 \\
0 & 0 & 1 & 0 & 0 \\
X_{51} & X_{53} & X_{54} & 0 & -X_{26}\\
0 & 0 & 0 & 1 & 0\\
X_{71} & X_{73} & 0 & -X_{54} & -X_{24}\\
X_{81} & 0 & -X_{73} & -X_{53} & -X_{23}\\
0 & 0 & 0 & 0 & 1\\
0 & -X_{81} & -X_{71} & -X_{51} & -X_{21}\\
\end{array}\right)\end{equation}
The origin of the affine space $\affinev^\gb$ , namely the point at which all $X_{(r,c)}$ vanish,
corresponds clearly to $e_\gb$. The tangent space to $\miso$ at $e_\gb$ can therefore
be identified with the affine space $\affinev^\gb$ with co-ordinate functions $X_{(r,c)}$.

\subsection{The ideal of the tangent cone to $X_\ga^\gc$ at $e_\gb$}\label{ss.ideal-of-tgtcone}
Set $Y_\ga^\gc(\gb):=X_\ga^\gc\cap\affinev^\gb$. From \cite{La-Se-II} we can deduce a set of generators for the ideal $I$ of functions on $\affinev^\gb$ vanishing on $Y_\ga^\gc(\gb)$ (see for example \cite{sthesis}, \S 3.2.2 for the special case of Schubert varieties). We recall this result now.

In the matrix~\ref{eq.matrix}, columns are numbered by the entries of $\gb$, the
rows by $\{1,\ldots,2d\}$. For $\theta\in\id$, consider the submatrix given by the rows
numbered $\theta\setminus\gb$ and columns numbered $\gb\setminus\theta$. Such a submatrix being of even
size and skew-symmetric along the anti-diagonal, we can define its \textit{Pfaffian} (see \S 3 of \cite{Ra-Up2}). Let $f_{\theta,\gb}$ denote this Pfaffian. We have
\begin{equation}\label{eq.ideal}
I=\left(f_{\tau,\gb}\st \tau\in\id, \ga\not\leq\tau\ or\ \tau\not\leq\gc\right).
\end{equation}
We are interested in the tangent cone to $X_\ga^\gc$ at $e_\gb$ or, what is the
same, the tangent cone to $Y_\ga^\gc(\gb)\subseteq\affinev^\gb$ at the origin. Observe that $f_{\tau,\gb}$ is a
homogeneous polynomial of degree the $\gb$-degree of $\tau$, where the $\gb$-degree of $\tau$
is defined as one half of the cardinality of $\gb\setminus\tau$. Because of this, $Y_\ga^\gc(\gb)$ itself
is a cone and so equal to its tangent cone. The ideal of the tangent cone to $X_\ga^\gc$ at $e_\gb$ is
therefore the ideal $I$ in equation~\ref{eq.ideal}.

\subsection{The term order}\label{ss.termorder}
We now specify the term order $\torder$ on monomials in the coordinate functions (of the tangent space to $\miso$ at the torus fixed point $e_\gb$) with respect to which the
initial ideal of the ideal $I$ of the tangent cone is to be taken.

Let $>$ be a total order on~$\orootsb$ satisfying all of the following $6$ conditions:
\begin{itemize}
\item 
$\mu>\nu$ if $\mu\in\ortho\posb$,  $\nu\in\orootsb\setminus\ortho\posb$,
and the row indices of $\mu$ and $\nu$ are equal.
\item
$\mu>\nu$ if $\mu\in\ortho\posb$,  $\nu\in\ortho\posb$,
the row indices of $\mu$ and $\nu$ are equal,  and the column index
of~$\mu$ exceeds that of~$\nu$.
\item 
$\mu>\nu$ if $\mu\in\ortho\posb$,  $\nu\in\orootsb$ and the row index of $\mu$
is less than that of~$\nu$.
\item 
$\mu>\nu$ if $\mu\in\orootsb\setminus\ortho\posb$,  $\nu\in\ortho\posb$,
and the column indices of $\mu$ and $\nu$ are equal.
\item
$\mu>\nu$ if $\mu\in\orootsb\setminus\ortho\posb$,  $\nu\in\orootsb\setminus\ortho\posb$,
the column indices of $\mu$ and $\nu$ are equal,  and the row index
of~$\mu$ exceeds that of~$\nu$.
\item 
$\mu>\nu$ if $\mu\in\orootsb\setminus\ortho\posb$,  $\nu\in\orootsb$ and the column index of $\mu$
is less than that of~$\nu$.
\end{itemize}
Note here that the first $3$ conditions above are the same as the conditions put on the total order $>_1$ as mentioned in \S 1.6 of \cite{Ra-Up2}. Recall that in the paper \cite{Ra-Up2}, initial ideals of ideals of tangent cones at torus fixed points to Schubert varieties in Orthogonal Grassmannians were computed, the paper \cite{Ra-Up2} did not deal with Richardson varieties. The last $3$ conditions above arise in this paper as an addition to the $3$ conditions put on the total order $>_1$ (as mentioned in \S 1.6 of \cite{Ra-Up2}), because here we are dealing with Richardson varieties in $\miso$. 

Let~$\torder$ be the term order on monomials in~$\orootsb$
(terminology as in~\cite[pages~329,~330]{ebud}) given by:
\begin{itemize}
\item the homogeneous lexicographic order with respect to~$>$. 
\end{itemize}

\brem\label{r.termorder}
The total order $>$ on~$\orootsb$ satisfying the 6 properties mentioned above can be realized as a concrete total order on $\orootsb$ if we put the following extra condition on it :
\begin{center}
Let $r(\mu), r(\nu), c(\mu), c(\nu)$ denote the row index of $\mu$, the row index of $\nu$, the column index of $\mu$, and the column index of $\nu$ respectively. If $r(\mu)<r(\nu)$, $\mu\in\orootsb\setminus\ortho\posb$, $\nu\in\ortho\posb$ and $c(\nu)<c(\mu)$, then
\begin{itemize}
 \item $\nu>\mu$ when $(r(\nu),c(\mu))\notin\posb$ AND $\mu>\nu$ when $(r(\nu),c(\mu))\in\posb$.
\end{itemize}
\end{center}
\erem\label{r.termorder}

\subsection{Extended $v$-chains and associated elements of $\id$}\label{ss.extd-v-chains}
For this subsection, let $v$ be a arbitrarily fixed element of $I(d,2d)$ (not necessarily an element of $\id$, unless otherwise stated). For elements $\lambda=(R,C),\mu=(r,c)$ of $\androotsv$, we write $\lambda>\mu$ if $R>r$ and $C<c$ (Note that these are strict inequalities). A sequence $\lambda_1>\cdots>\lambda_k$ of elements of $\rootsv$ is called an \textit{extended} $v$-chain. The points indicated by solid circles in the picture in \S~\ref{ss.basicnot} form an extended $v$-chain. Note that an extended $v$-chain can also be empty. Letting $C$ to be an extended $v$-chain, we define $C^+:=C\cap\posv$ and $C^-:=C\cap(\rootsv\setminus\posv)$. We call $C^+$ (resp. $C^-$) to be the \em{positive} (resp. \em{negative}) parts of the extended $v$-chain $C$. We call an extended $v$-chain $C$ to be \em{positive} (resp. \em{negative}) if $C=C^+$ (resp. $C=C^-$). The extended $v$-chain $C$ is called \textit{non-vanishing} if at least one of its positive or negative parts is non-empty. Clearly then, every non-empty extended $v$-chain is non-vanishing. Note that if we specialize to the case when $v\in\id$, then whatever is called a $v$-chain in \S 2.2.1 of \cite{sthesis} is a positive extended $v$-chain over here. To every extended $v$-chain $C$, we will now associate $2$ subsets $\diagv_C(+)$ and $\diagv_C(-)$ of $\diagv$ (each of even cardinality), but for that we first need to fix some notation and recall certain terminology from \cite{sthesis}. 

\bdefn\label{defn.PrandPro}
\textbf{$Pr$ and $Pro$}: Given any subset $D$ of $\posv$, let us denote by $Pr(D)$ the multiset (that means, counting multiplicities) of the projections (both vertical and horizontal, as defined in \S 5.3.1 of \cite{sthesis}) of all its elements on $\diagv$. For $\lambda=(r,c)$ in $\androotsv$, define $\lambda^\hash:=(c^*,r^*)$. The involution $\lambda\mapsto\lambda^\hash$ is just
the reflection with respect to the diagonal $\diagv$. For a subset $\mathfrak{E}$ of
$\andposv$, the symbol $\mathfrak{E}^\hash$ has the obvious meaning. We call $\mathfrak{E}$ \textit{symmetric} if $\mathfrak{E}=\mathfrak{E}^\hash$. Given any symmetric subset $E$ of $\andposv$, let us denote by $E(top)$ the set $E\cap\posv$ and by $E(diag)$, the set $E\cap\diagv$, and by $Pro(E)$ the multiset formed by taking the union of the subset $E(diag)$ with the multiset $Pr(E(top))$. Let us make the definition of $Pro(E)$ more precise: The multiplicity with which any element occurs in the multiset $Pro(E)$ is equal to the sum of the multiplicities with which the element occurs in the subset $E(diag)$ and the multiset $Pr(E(top))$. So for any symmetric subset $E$ of $\andposv$, $Pro(E)$ is a multiset consisting of elements from the diagonal. Similarly for any subset $D$ of $\posv$, $Pr(D)$ is also a multiset consisting of elements from the diagonal. 
\edefn

If we take $v$ to be in $\id$, we can recall from \S 5.3 of \cite{sthesis} the definition of the monomial $\mon_C$ attached to a $v$-chain $C$ (Note that a $v$-chain in \cite{sthesis} is a positive extended $v$-chain over here). Note that even if we take $v$ to be in $I(d,2d)$ (and not merely in $\id$) and define $\mon_C$ for any positive extended $v$-chain $C$ exactly in the same way as we did in \S 5.3 of \cite{sthesis}, there is no logical inconsistency. Hence we extend the definition of $\mon_C$ to any positive extended $v$-chain $C$ where $v\in I(d,2d)$. Clearly $\mon_C$ is a symmetric subset of $\andposv$. Hence the multiset $Pro(\mon_C)$ is well defined for any positive extended $v$-chain $C$ where $v\in I(d,2d)$. 

\bdefn\label{defn.flip}
\textbf{The flip map $F$}: For any $v\in I(d,2d)$ and any element $\lambda=(r,c)\in\androotsv$, let $F(\lambda)$ be the element of $\rootsvstar$ given by $F(\lambda):=(c,r)$. So $F$ is an invertible map from $\androotsv$ to $\rootsvstar$ (note here that if $v\in\id$, then $v^*$ need not always belong to $\id$), let us denote the inverse of $F$ by $F^{-1}$. The map  $F$ naturally induces an invertible map from the set of all multisets in $\androotsv$ to the set of all multisets in $\rootsvstar$. We continue to call the induced map also as $F$ and its inverse as $F^{-1}$. 
\edefn
 
\bdefn\label{defn.diagvC+-}
\textbf{The subsets $\diagv_C(+)$ and $\diagv_C(-)$ of $\diagv$}: Given any extended $v$-chain $C$, we will now associate $2$ subsets $\diagv_C(+)$ and $\diagv_C(-)$ of $\diagv$ (each of even cardinality) to it as mentioned towards the beginning of this subsection. Let 
\begin{center}
$\diagv_C(+):=\left\{\begin{array}{ccc}
Pro(\mon_C) & if & C\ is\ positive\\
F^{-1}(Pr(F(C))\setminus_m Pro(\mon_{F(C)})) & if & C\ is\ negative\\
Pro(\mon_{C^+}) & if & C\ is\ non-vanishing\\ \end{array}\right.$
\end{center}
Similarly, let 
\begin{center}
$\diagv_C(-):=\left\{\begin{array}{ccc}
Pr(C)\setminus_m Pro(\mon_C) & if & C\ is\ positive\\
F^{-1}(Pro(\mon_{F(C)})) & if & C\ is\ negative\\
F^{-1}(Pro(\mon_{F(C^-)})) & if & C\ is\ non-vanishing\\ \end{array}\right.$
\end{center} 
It is an easy exercise to check that $\diagv_C(+)$ and $\diagv_C(-)$ thus defined are actually subsets (not multisets) of $\diagv$ and that each of them has even cardinality. 
\edefn

\bdefn\label{defn.wC+-}
\textbf{Elements of $\id$ associated to $\diagv_C(+)$ and $\diagv_C(-)$}: For this definition, we let $v$ to be an arbitrary element in $\id$. Note that given any subset $S$ of $\diagv$ of even cardinality, we can naturally associate an element of $\id$ to it by removing those entries from $v$ which appear as column indices in the elements of the set $S$ and then adding to it the row indices of all the elements of $S$. It is easy to check that the resulting element actually belongs to $\id$. We denote the resulting element by $\id(S)$. If $S$ is empty, then $\id(S)$ is taken to be $v$ itself.

Let $w_C^+(v):=\id(\diagv_C(+))$ and $w_C^-(v):=\id(\diagv_C(-))$. These are the two elements of $\id$ that we can naturally associate to the subsets $\diagv_C(+)$ and $\diagv_C(-)$ of $\diagv$.
\edefn

\subsection{The main theorem and a strategy of the proof}\label{ss.mainthm+strategy}
Recall that the ideal of the tangent cone to $X_\ga^\gc$ at $e_\gb$ is
the ideal $I$ given by equation~\ref{eq.ideal}, that is,\\
\begin{equation}
I=\left(f_{\tau,\gb}\st \tau\in\id, \ga\not\leq\tau\ or\ \tau\not\leq\gc\right).
\end{equation}
Let $\torder$ be as in \ref{ss.termorder}. For any element $f\in I$, let $\init f$ denote the initial term of $f$ with respect to the term order $\torder$. We define $\init I$ to be the ideal $\langle\init f\st f\in I\rangle$ inside the polynomial ring $P:=\field[X_{(r,c)}\st (r,c)\in\orootsb]$. For any monomial $U$ in $\orootsb$, let us denote by $X_U$ the product of all the elements $X_{(r,c)}$ where $(r,c)$ runs over all elements in $U$. 

Let $Chains_\ga^\gc(\gb)$ denote the set $\{X_C\st C\ is\ a\ non-vanishing\ extended\ \gb-chain\ in\ \orootsb\ such\ that\ either\ (i)\ or\ (ii)\ of\ \ref{eq.chains}\ holds\}$.
\begin{equation}\label{eq.chains}
(i) C^-\ is\ non-empty\ and\ \ga\not\leq w_{C^-}^-(\gb).
(ii) C^+\ is\ non-empty\ and\ w_{C^+}^+(\gb)\not\leq\gc.
\end{equation}

\noindent The main result of this paper is the following:---

\begin{thm}\label{t.main}
$\init I=\langle Chains_\ga^\gc(\gb)\rangle$. 
\end{thm}

\brem\label{r.initI}
It follows from the statement of Theorem~\ref{t.main} above that: The set of all monomials in $\orootsb$ which contain at least one extended $\gb$-chain $C$ such that $X_C\in Chains_\ga^\gc(\gb)$, form a vector space basis of the initial ideal $\init I$ over the field $\field$. In the special case when the Richardson variety is a Schubert variety, it is easy to see that the previous statement of this remark says exactly what has been said in the main theorem (Theorem 1.8.1) of \cite{Ra-Up2}.  \erem

We now briefly sketch the proof of Theorem~\ref{t.main} (omitting details, which can be found in Section \ref{s.grobner}). In order to introduce the main combinatorial objects of interest and outline a strategy of the proof, we will first need to prove that the set $Chains_\ga^\gc(\gb)\subseteq\init I$, and this proof will follow from whatever is said in Remark~\ref{r.Chains-in-initI} below.
\brem\label{r.Chains-in-initI}
Let $C$ be a non-vanishing extended $\gb$-chain in $\orootsb$ such that $X_C\in Chains_\ga^\gc(\gb)$. If $C^+$ is non-empty and $w_{C^+}^+(\gb)\not\leq\gc$, then it can be proved that $X_{C^+}\in\init I$, the proof being exactly the same as that in \S 4 of \cite{Ra-Up2}. Then since $\init I$ is an ideal and $X_C=X_{C^-}X_{C^+}$, it follows that $X_C\in\init I$. 

If $C^-$ is non-empty and $\ga\not\leq w_{C^-}^-(\gb)$, look at $F(C^-)$ where $F$ is the \textit{flip} map as defined in \S \ref{defn.flip} from the set of all multisets in $\rootsb$ to the set of all multisets in $\rootsbstar$. Then $F(C^-)$ is a positive extended $\gb^*$-chain in $\ortho\rootsbstar$. We need to prove that $X_{C}\in\init I$, for which it is enough to show that $X_{C^-}\in\init I$. To prove that $X_{C^-}\in\init I$, we will proceed in a way equivalent to the proof done in \S 4 of \cite{Ra-Up2}. But there is a subtle difference between what is proved in \S 4 of \cite{Ra-Up2} and what we are going to prove here, namely: Whatever was proved in \S 4 of  \cite{Ra-Up2} can be rephrased in the language of this paper as `Every positive extended $\gb$-chain $D$ satisfying the property that $w_{D}^+(\gb)\not\leq\gc$ belongs to the initial ideal of the ideal of the tangent cone', but here we are going to prove that `Every negative extended $\gb$-chain $D$ satisfying the property that $\ga\not\leq w_{D}^-(\gb)$ belongs to the initial ideal of the ideal of the tangent cone'. 

Because of this subtle difference, we need to construct certain gadjets for negative extended $\gb$-chains, which will play role similar to the role of the objects like the \textit{new forms}, $Proj$ and $Proj^e$ corresponding to positive extended $\gb$-chains (For positive extended $\gb$-chains, such objects are already defined in \cite{Ra-Up2}). This construction is given in the following paragraph.

Consider the positive extended $\gb^*$-chain $F(C^-)$. We can construct \textit{new forms}, $Proj$ and $Proj^e$ for $F(C^-)$ in the same way as they were constructed in \cite{Ra-Up2}, note here the fact that $\gb^*$ may or may not belong to $\id$ does not really effect the construction of the new forms, $Proj$ and $Proj^e$ for $F(C^-)$. Then we apply the map $F^{-1}$ to these objects constructed for $F(C^-)$ , the resulting objects are the analogues of the \textit{new forms}, $Proj$ and $Proj^e$ for the negative extended $\gb$-chain $C^-$. We apply similar treatment to any other monomial related to $F(C^-)$ that we happen to encounter if we replace the `$v$-chain $A$' in \S 4.2 of \cite{Ra-Up2} by `the positive extended $\gb^*$-chain $F(C^-)$'. 

In \S 2.4 of \cite{Ra-Up2}, an element $y_E$ of $\id$ corresponding to any $v$-chain $E$ (the notion of a $v$-chain being as in \S 1.7 of \cite{Ra-Up2}) has been defined. The analogous element of $\id$ for the negative extended $\gb$-chain $C^-$ (We call it $y_{C^-}$ here) can be obtained from $F^{-1}(Proj^e(F(C^-)))$ by following the natural process: \textit{the column indices of elements of $F^{-1}(Proj^e(F(C^-)))$ occur as members of $\gb$; these are replaced by the row indices to obtain the desired element of $\id$ for $C^-$}. It is easy to check that $y_{C^-}$ belongs to $\id$ and that $y_{C^-}\leq w_{C^-}^-(\gb)\leq\gb$. Since we already have that $\ga\not\leq w_{C^-}^-(\gb)$, it follows that $\ga\not\leq y_{C^-}$. These facts about $y_{C^-}$ will be needed to produce an analogue of the main proof of \cite{Ra-Up2} in our present case. To be more precise, these facts about $y_{C^-}$ give the analogues of Propositions 2.4.1 and 2.4.2 of \cite{Ra-Up2} and these two propositions had been used quite crucially inside the main proof of \cite{Ra-Up2}.  

With all these analogues constructed for negative extended $\gb$-chains, we can proceed in an `equivalent' manner (Here, by `equivalent' we mean: keeping track of the subtle difference as mentioned above and working accordingly) as in the paper \cite{Ra-Up2} and end up proving the desired fact, viz., $X_{C^-}\in\init I$.  
\erem
 
Since $Chains_\ga^\gc(\gb)\subseteq\init I$, we have $\langle Chains_\ga^\gc(\gb)\rangle\subseteq\init I$. To prove Theorem~\ref{t.main}, we now need to show that $\langle Chains_\ga^\gc(\gb)\rangle\supseteq\init I$. For this, it suffices to show that in any degree, the number of monomials of $\langle Chains_\ga^\gc(\gb)\rangle$ is $\geq$ the number of monomials of $\init I$. Or equivalently, it suffices to show that in any degree, the number of monomials of $P/\langle Chains_\ga^\gc(\gb)\rangle$ is $\leq$ the number of monomials of $P/ \init I$. 

Recall from \S~\ref{ss.ideal-of-tgtcone} the affine patch $Y_\ga^\gc(\gb)(:=X_\ga^\gc\cap\affinev^\gb)$ of the Richardson variety $X_\ga^\gc$. The following is a well known result (see \cite{Br,La-Go}, for
instance).
\begin{thm}\label{t.coord-Ygagcgb} $\field[Y_\ga^\gc(\gb)]=P/I$ where $P=\field[X_{(r,c)}\st (r,c)\in\orootsb]$ and $I$ is as in equation~\ref{eq.ideal}.
\end{thm}
 
Both the monomials of $P/\init I$ and the standard monomials on $Y_\ga^\gc(\gb)$ form a basis for $P/I$, and thus agree in cardinality in any degree. Therefore, to prove that in any degree, the number of monomials of $P/\langle Chains_\ga^\gc(\gb)\rangle$ is $\leq$ the number of monomials of $P/\init I$, it suffices to give a degree-preserving injection from the set of all monomials in $P/\langle Chains_\ga^\gc(\gb)\rangle$ to the set of all standard monomials on $Y_\ga^\gc(\gb)$. We construct such an injection, the \textbf{Orthogonal-bounded-RSK (OBRSK)}, from an indexing set of the former to an indexing set of the later. These indexing sets are given in the table of figure ~\ref{f.i.main_objects}.\\

\begin{figure}[!h]
\begin{center}
\begin{tabular}{|l|l|}\hline
\multicolumn{1}{|c|}{\slshape Set of elements in $P=\field[X_{(r,c)}\st (r,c)\in\orootsb]$}&
\multicolumn{1}{c|}{\slshape Indexing set}\\
\hline\hline
&{pairs of non-vanishing skew-symmetric multisets}\\
\raisebox{.5em}[0pt]{monomials of $P/\langle Chains_\ga^\gc(\gb)\rangle$}& on $\BBZ$ bounded by $T_\ga, W_\gc$\\
\hline
&{non-vanishing skew-symmetric notched bitableaux}\\
\raisebox{.5em}[0pt]{standard monomials on $Y_\ga^\gc(\gb)$}&on $\BBZ$
bounded by $T_\ga, W_\gc$
\\
\hline
\end{tabular}
\end{center}
\caption{\label{f.i.main_objects}Two subsets of the ring $P=\field[X_{(r,c)}\st (r,c)\in\orootsb]$ and
their indexing sets}
\end{figure}

\noindent In Sections \ref{s.ss_lex_arrays},
\ref{s.ss_notched_bitab}, \ref{s.obrsk}, and \ref{s.gobrsk_props}, we develop the necessary things and finally also 
define pairs of non-vanishing skew-symmetric multisets on $\BBZ$ bounded by $T_\ga, W_\gc$,
non-vanishing skew-symmetric notched bitableaux on $\BBZ$ bounded by
$T_\ga, W_\gc$, and the injection \textbf{OBRSK} from the former to the
latter. In Section \ref{s.grobner}, we prove that these two
combinatorial objects are indeed indexing sets for the monomials
of $P/\langle Chains_\ga^\gc(\gb)\rangle$ and the standard monomials on
$Y_\ga^\gc(\gb)$ respectively, and use this to prove Theorem~\ref{t.main}.

\section{Skew-symmetric Notched Bitableaux}\label{s.ss_notched_bitab}
This section onwards, the terminology and notation of \S 4 and \S 5 of \cite{Kr-bkrs} will be in force. Recall the definition of a \textbf{semistandard} notched bitableau from \S 5 of \cite{Kr-bkrs}. 

\bdefn\label{defn.dual-in-bitableau}
\textbf{Dual of an element with respect to a semistandard notched bitableau}: Let $(P,Q)$ be a semistandard notched bitableau. Let $p_{i,j}$ (resp. $q_{i,j}$) denote the entry in the $i$-th row and $j$-th column of $P$ (resp. $Q$). For any row number $i$ of $P$ (or of $Q$), let  $k_i$ denote the total number of entries in the $i$-th row of $P$ (or $Q$). We call the entry $q_{i,k_i+1-j}$ of $Q$ to be the \textbf{dual of} the entry $p_{i,j}$ of $P$ \textbf{with respect to} $(P,Q)$. Similarly, we call the entry $p_{i,k_i+1-j}$ of $P$ to be the \textbf{dual of} the entry $q_{i,j}$ of $Q$ \textbf{with respect to} $(P,Q)$.\edefn 

Note that any entry of $P$ or $Q$ can be identified uniquely by specifying $4$ coordinates, namely: the entry $x$, the tableau $A$ ($A=P$ or $Q$) in which the entry lies, the row number $i$ of the entry in the tableau $A$, and the column number $j$ of the entry in the tableau $A$. Let $Set(P,Q)$ denote the set of all $4$-tuples of the form $(x,A,i,j)$. Given any $4$-tuple $(x,A,i,j)\in Set(P,Q)$, let us denote by $D_{(P,Q)}(x,A,i,j)$ the dual of $x$ with respect to $(P,Q)$ as defined in \ref{defn.dual-in-bitableau} above. For $(x,A,i,j), (x',A',i',j')\in Set(P,Q)$, we say that $(x,A,i,j)\leq (x',A',i',j')$ if $x\leq x'$, and similarly for strict inequality and equality.

A semistandard notched bitableau $(P,Q)$ is said to be \textbf{Skew-symmetric} if the following 2 conditions are satisfied simultaneously:---\\
(i) The bitableau $(P,Q)$ should be of {\em even} size, that is, the number of elements in each row of $P$ and $Q$ should be even.

\noindent (ii) If $(x,A,i,j), (x',A',i',j')\in Set(P,Q)$ are such that $(x,A,i,j)\leq (x',A',i',j')$, then $D_{(P,Q)}(x,A,i,j)\geq D_{(P,Q)}(x',A',i',j')$. Moreover, $(x,A,i,j)<(x',A',i',j')$ implies $D_{(P,Q)}(x,A,i,j)>D_{(P,Q)}(x',A',i',j')$ and $(x,A,i,j)=(x',A',i',j')$ implies $D_{(P,Q)}(x,A,i,j)=D_{(P,Q)}(x',A',i',j')$.

\noindent Property (ii) above will be henceforth referred to as \textbf{the duality property} associated to the Skew-symmetric notched bitableau $(P,Q)$. Note that a \textbf{Skew-symmetric} notched bitableau is \textbf{a semistandard notched bitableau} by default. The \textbf{degree} of a Skew-symmetric notched bitableau $(P,Q)$ is the total number of boxes in $P$ (or $Q$). The notions of \textbf{negative, positive} and \textbf{nonvanishing} Skew-symmetric notched bitableau remain the same as in \S 5 of \cite{Kr-bkrs}. The notion of a Skew-symmetric notched bitableau $(P,Q)$ being \textbf{bounded by} $\mathbf{T, W}$ (where $T,W$ are subsets of $\bN^2$), as well as the notion of \textbf{negative} and \textbf{positive parts} of a Skew-symmetric notched bitableau $(P,Q)$ remain the same as they were in \S 5 of \cite{Kr-bkrs}. 

If $(P,Q)$ is a nonvanishing skew-symmetric notched bitableau,
define $\gi(P,Q)$ to be the notched bitableau obtained by
reversing the order of the rows of $(Q,P)$. One checks that
$\gi(P,Q)$ is a nonvanishing skew-symmetric notched bitableau. The
map $\gi$ is an involution, and it maps negative skew-symmetric
notched bitableaux to positive ones and visa-versa. Thus $\gi$
gives a bijective pairing between the sets of negative and
positive skew-symmetric notched bitableaux. 
\section{Skew-symmetric lexicographic arrays}\label{s.ss_lex_arrays}
By a two-row \textbf{lexicographic array}, we mean: A two-row array of positive integers
\begin{equation}\pi= \left(\begin{array}{ccc}
\beta_1 & \cdots & \beta_t\\
\alpha_1 & \cdots & \alpha_t\\ \end{array}\right)\end{equation} 
such that $\beta_k\geq\beta_{k+1}\ \forall\ k$, and if $\beta_k=\beta_{k+1}$, then $\alpha_k\geq\alpha_{k+1},\ k=1,\ldots,t-1$. 

Given a lexicographic array $\pi$, let $\pi^t$ denote the array (not necessarily lexicographic) obtained by switching the two rows of $\pi$. We call the array $\pi^t$ to be the \textbf{transpose of the array $\pi$}.

Consider a pair $\{\pi_1,\pi_2\}$ of two-row arrays (not necessarily lexicographic) where both $\pi_1$ and $\pi_2$ are of the same \textbf{degree} (say, $t$, the degree of a two-row array is the number of columns in the array) of positive integers where $\pi_1$ and $\pi_2$ are given by:---

\begin{equation}\label{eq.arrays}\pi_1= \left(\begin{array}{ccc}
b_1 & \cdots & b_t\\
a_1 & \cdots & a_t\\ \end{array}\right)\ and\ \pi_2= \left(\begin{array}{ccc}
c_1 & \cdots & c_t\\
d_1 & \cdots & d_t\\ \end{array}\right)\end{equation} 

We call the lower row of the array $\pi_1$ the \textbf{$a$-row}, the upper row of the array $\pi_1$ the \textbf{$b$-row}, the lower row of the array $\pi_2$ the \textbf{$d$-row} and, the upper row of the array $\pi_2$ the \textbf{$c$-row}. Any entry in the pair $\{\pi_1,\pi_2\}$ of arrays can be identified uniquely by specifying $3$ coordinates: the row $\mathcal{r}$ of $\{\pi_1,\pi_2\}$ in which the entry lies ($\mathcal{r}=a,b,c\ or\ d$), the position $i$ (counting from left to right) of the entry in the row $\mathcal{r}$ and, the value $\mathcal{v}(\mathcal{r},i)$ of the entry sitting in the $i$-th position of the row $\mathcal{r}$.   
 
Set $S_{\pi_1,\pi_2}:=\{x|x=(\mathcal{r},i,\mathcal{v}(\mathcal{r},i)),\mathcal{r}\in\{a,b,c,d\},i\in\{1,\ldots,t\}\}$. For any $x\in S_{\pi_1,\pi_2}$, let 
\begin{center}
$D_{\pi_1,\pi_2}(x):=\left\{\begin{array}{cccc}
(c,t+1-i,\mathcal{v}(c,t+1-i)) & if & x=(a,i,\mathcal{v}(a,i)) & \\
(d,t+1-i,\mathcal{v}(d,t+1-i)) & if & x=(b,i,\mathcal{v}(b,i)) & \forall\ i\in\{1,\ldots,t\}\\
(a,t+1-i,\mathcal{v}(a,t+1-i)) & if & x=(c,i,\mathcal{v}(c,i)) & \\
(b,t+1-i,\mathcal{v}(b,t+1-i)) & if & x=(d,i,\mathcal{v}(d,i)) & \\ \end{array}\right.$
\end{center} 

We call $D_{\pi_1,\pi_2}(x)$ the \textbf{Dual of $x$ with respect to the pair $\{\pi_1,\pi_2\}$ of arrays}. Note that for every $x\in S_{\pi_1,\pi_2}$, we have $D_{\pi_1,\pi_2}(x)\in S_{\pi_1,\pi_2}$. For any two elements $x,x'\in S_{\pi_1,\pi_2}$ where $x=(\mathcal{r},i,\mathcal{v}(\mathcal{r},i))$ and $x'=(\mathcal{r}',i',\mathcal{v}(\mathcal{r}',i'))$, we say that $x\leq x'$ if $\mathcal{v}(\mathcal{r},i)\leq \mathcal{v}(\mathcal{r}',i')$. Similar notion applies to saying that $x<x'$ or $x=x'$. 

The above pair $\{\pi_1,\pi_2\}$ of arrays is said to be \textbf{Skew-symmetric lexicographic} if the following conditions are satisfied simultaneously:---

(i) $\pi_1$ is a lexicographic array.

(ii) $\pi_2^t$ is a lexicographic array.

(iii) $a_i<d_{t+1-i}\forall\ i\in\{1,\ldots,t\}$.

(iv) $b_i<c_{t+1-i}\forall\ i\in\{1,\ldots,t\}$.

(v) For any $x,y\in S_{\pi_1,\pi_2}$, if $x\leq y$, then $D_{\pi_1,\pi_2}(x)\geq D_{\pi_1,\pi_2}(y)$. Also strict inequality on one side implies strict inequality on the other side, in the sense that if $x<y$, then $D_{\pi_1,\pi_2}(x)>D_{\pi_1,\pi_2}(y)$. And similarly for equality. ($\rightarrow$ This property is called the \textbf{Duality Property} associated to the pair $\{\pi_1,\pi_2\}$ of Skew-symmetric lexicographic arrays.)

(vi) For each $k\in\{1,\ldots,t\}$, if $a_k<b_k$, then $d_{t+1-k}<c_{t+1-k}$, and if $a_k>b_k$, then $d_{t+1-k}>c_{t+1-k}$.

For any pair $\{\pi_1,\pi_2\}$ of Skew-symmetric lexicographic arrays, we define the \textbf{degree of the pair} to be $2$ times the degree of $\pi_1$ (or of $\pi_2$, they are the same). A pair $\{\pi_1,\pi_2\}$ of Skew-symmetric lexicographic arrays is said to be \textbf{negative} if $a_k<b_k,k=1,\ldots,t$, \textbf{positive} if $a_k>b_k,k=1,\ldots,t$, and \textbf{non-vanishing} if $a_k\neq b_k,k=1,\ldots,t$. Note that condition (vi) above will imply that if $a_k<b_k\forall k=1,\ldots,t$, then $d_{t+1-k}<c_{t+1-k}\forall k=1,\ldots,t$. Similarly, if $a_k>b_k\forall k=1,\ldots,t$, then $d_{t+1-k}>c_{t+1-k}\forall k=1,\ldots,t$ and if $a_k\neq b_k\forall k=1,\ldots,t$, then $d_{t+1-k}\neq c_{t+1-k}\forall k=1,\ldots,t$. 

Let $\{\pi_1,\pi_2\}$ be a pair of non-vanishing Skew-symmetric lexicographic arrays. Let us denote by $\pi_1^-$ (resp. $\pi_1^+$) the lexicographic array consisting of those columns of $\pi_1$ such that $a_i<b_i$ (resp. $a_i>b_i$). Let us denote by $\pi_2^-$ (resp. $\pi_2^+$) the lexicographic array consisting of those columns of $\pi_2$ such that $d_i<c_i$ (resp. $d_i>c_i$). We call $\{\pi_1^-,\pi_2^-\}$ and $\{\pi_1^+,\pi_2^+\}$ to be the \textbf{negative} and \textbf{positive parts} respectively of the pair $\{\pi_1,\pi_2\}$. Note here that because of condition (vi) above, $\pi_1^-$ and $\pi_2^-$ will have the same degree, and the same holds true for $\pi_1^+$ and $\pi_2^+$. It is easy to see now that both the pairs $\{\pi_1^-,\pi_2^-\}$ and $\{\pi_1^+,\pi_2^+\}$ of arrays are Skew-symmetric lexicographic in their own right. 

Given a lexicographic array $\pi$, define $l(\pi)$ to be the lexicographic array obtained by first switching the two rows of $\pi$ and then rearranging the columns so that the new array is lexicographic. Let $l^t$ be a map from the set of all lexicographic arrays to itself given by first switching the two rows of a given lexicographic array $\pi$, and then rearranging the columns so that the resulting array's transpose becomes lexicographic.

We now define a map $L$ from the set of all pairs of Skew-symmetric lexicographic arrays to itself, as follows:---
\begin{center}
$L(\{\pi_1,\pi_2\}):=\{l(\pi_1),l^t(\pi_2)\}$
\end{center}  
It is easy to check that the above map $L$ is well-defined, it is an involution, and it maps pairs of negative Skew-symmetric lexicographic arrays to positive ones, and vice-versa. Thus $L$ gives a bijective pairing between the set of all pairs of negative Skew-symmetric lexicographic arrays and the set of all pairs of positive Skew-symmetric lexicographic arrays.

\section{The Orthogonal-Bounded RSK Correspondence}\label{s.obrsk}
We next define the \textbf{Orthogonal bounded RSK correspondence, OBRSK} a function which maps a pair of negative Skew-symmetric lexicographic arrays to a negative Skew-symmetric notched bitableau. Let $\{\pi_1,\pi_2\}$ be a pair of negative Skew-symmetric lexicographic arrays whose entries are labelled as in \ref{eq.arrays}. We inductively form a sequence of notched
bitableaux $(P^{(0)},Q^{(0)})$, $(P^{(1)},Q^{(1)}),$
$\ldots,(P^{(t)},Q^{(t)})$, such that each $(P^{(i)},Q^{(i)})$ is of even size and $P^{(i)}$ is semistandard on
$b_i$ for every $i=1,\ldots,t$, as follows:\\
\begin{quote}
Let $(P^{(0)},Q^{(0)})=(\es,\es)$, and let $b_0=b_1$. Assume
inductively that we have formed $(P^{(i)}, Q^{(i)})$, such that the notched bitableau $(P^{(i)}, Q^{(i)})$ is of even size, $P^{(i)}$ is semistandard on $b_i$, and thus on $b_{i+1}$, since
$b_{i+1}\leq b_i$. 

Let us first fix some notation and terminology. Let $p_{kj}^{(i)}$ (resp. $q_{kj}^{(i)}$) denote the entry in the $k$-th row and $j$-th column of $P^{(i)}$ (resp. $Q^{(i)}$). Let $2l_k^{(i)}$ denote the total number of entries (note that it is always even) in the $k$-th row of $P^{(i)}$ (or $Q^{(i)}$). 

Given an arbitrary notched tableau $P$, and any row number $k$ of $P$, we call the entry in the $j$-th box \textbf{(counting from left to right)} as the \textbf{Forward $j$-th entry of the $k$-th row of $P$}. Similarly, we call the entry in the $j$-th box \textbf{(counting from right to left)} of $P$ as the \textbf{Backward $j$-th entry of the $k$-th row of $P$}. 

It is now easy to see that the backward $j$-th entry of the $k$-th row of $Q^{(i)}$ is actually equal to the forward $(2l_k^{(i)}+1-j)$-th entry of $Q^{(i)}$. We now describe the \textbf{OBRSK} correspondence for the pair $\{\pi_1,\pi_2\}$ of negative Skew-symmetric lexicographic arrays as mentioned above in \ref{eq.arrays}. 

Perform the bounded insertion process $P^{(i)}\la{b_{i+1}}a_{i+1}$ as in \cite{Kr-bkrs}. In this finite-step process of bounded insertion, suppose that $a_{i+1}$ had bumped the `Forward $j_1$-th entry' of the $1$-st row of $P^{(i)<b_{i+1}}$, again say the `Forward $j_1$-th entry' of the $1$-st row of $P^{(i)<b_{i+1}}$ has bumped the `Forward $j_2$-th entry' of the $2$-nd row of $P^{(i)<b_{i+1}}$,... and so on until, at some point, a number is placed in a new box at the right end of some row of $P^{(i)<b_{i+1}}$, say this happens at the row number $K_{(i)}$ of $P^{(i)<b_{i+1}}$. Say that the entry of the new box (as mentioned in the previous statement) becomes the \textbf{Forward $j_{K_{(i)}}$-th entry of the $K_{(i)}$-th row of $P^{(i)}\la{b_{i+1}}a_{i+1}$}. Then we construct a new notched tableau (call it $Q^{(i)}\la{dual}c_{t-i}$) out of the tableau $Q^{(i)}$ and the entry $c_{t+1-(i+1)}(=c_{t-i})$ (note that $c_{t+1-(i+1)}$ is the same as $D_{\pi_1,\pi_2}(a_{i+1})$) of the array $\pi_2$ as follows:-- We let $c_{t-i}$ bump the `Backward $j_1$-th entry' of the $1$-st row of $Q^{(i)}$, then we let the `Backward $j_1$-th entry' of the $1$-st row of $Q^{(i)}$ bump the `Backward $j_2$-th entry' of the $2$-nd row of $Q^{(i)}$,... and so on until, at some point, a number is placed in a new box at the \textbf{Backward $j_{K_{(i)}}$-th position of the $K_{(i)}$-th row of $Q^{(i)}$}, shifting all entries in the Backward $1$-st ... upto (and including) the Backward $(j_{K_{(i)}}-1)$-th positions of the $K_{(i)}$-th row of $Q^{(i)}$ to the right by one box. We denote the resulting notched tableau by $Q^{(i)}\la{dual}c_{t-i}$. 

Note here that this integer $K_{(i)}$ can be equal to $1$ in some cases, then there are no `bumps' in the process of bounded insertion $P^{(i)}\la{b_{i+1}}a_{i+1}$. In such situations, look at the position of $a_{i+1}$ in the first row of the notched tableau $P^{(i)}\la{b_{i+1}}a_{i+1}$, say $a_{i+1}$ is the forward $j$-th entry of the $1$-st row of $P^{(i)}\la{b_{i+1}}a_{i+1}$. Then we place $c_{t-i}$ in a new box at the backward $j$-th position of the $1$-st row of $Q^{(i)}$, shifting all those entries which were in the Backward $1$-st ... upto (and including) the backward $(j-1)$-th positions of the $1$-st row of $Q^{(i)}$ to the right by one box. We denote the resulting notched tableau by $Q^{(i)}\la{dual}c_{t-i}$.

Basically, the idea is that whatever we did for the bounded insertion process producing $P^{(i)}\la{b_{i+1}}a_{i+1}$, we do a dual version of the same process on $Q^{(i)}$ with the integer $c_{t-i}$. Let us denote the resulting tableau by $Q^{(i)}\la{dual}c_{t-i}$. Note here that the tableaux $P^{(i)}\la{b_{i+1}}a_{i+1}$ and $Q^{(i)}\la{dual}c_{t-i}$ so constructed are of the same shape, but there exists one row in both of them in which the total number of entries is {\em odd}. We wanted to construct a notched bitableau $(P^{(i+1)}, Q^{(i+1)})$ inductively from $(P^{(i)}, Q^{(i)})$ which should be of even size. We make it possible in the following way:---

Recall the row number $K_{(i)}$ of $P^{(i)}$ (or of $Q^{(i)}$) at which the above mentioned insertion algorithm had stopped. Place $d_{t+1-(i+1)}$ ($=d_{t-i}$) in a new box at the rightmost end of the $K_{(i)}$-th row of $P^{(i)}\la{b_{i+1}}a_{i+1}$. We denote the resulting notched tableau by $P^{(i+1)}$. By the construction of $P^{(i)}\la{b_{i+1}}a_{i+1}$ (and as explained in \cite{Kr-bkrs}), we know that $P^{(i)}\la{b_{i+1}}a_{i+1}$ is semistandard on $b_{i+1}$. It is an easy exercise now to see that $P^{(i+1)}$ as constructed above will also continue to be semistandard on $b_{i+1}$, well the reason briefly is that $d_{t-i}$ is bigger than or equal to all entries of $P^{(i)}\la{b_{i+1}}a_{i+1}$ (This follows from the defining properties of the pair of negative Skew-symmetric lexicographic arrays $\{\pi_1,\pi_2\}$). After this, we place $b_{i+1}$ in a new box at the leftmost end of the $K_{(i)}$-th row of $Q^{(i)}\la{dual}c_{t-i}$, shifting all previously existing entries in the $K_{(i)}$-th row of $Q^{(i)}\la{dual}c_{t-i}$ to the right by one box. We denote the resulting notched tableau by $Q^{(i+1)}$. Clearly $P^{(i+1)}$ and $Q^{(i+1)}$ have the same shape. Now we have got hold of a notched bitableau $(P^{(i+1)}, Q^{(i+1)})$ which is of even size.  
\end{quote}
Then $OBRSK(\{\pi_1,\pi_2\})$ is defined to be $(P^{(t)},Q^{(t)})$.

In the process above, we write
$(P^{(i+1)},Q^{(i+1)})=(P^{(i)},Q^{(i)})\la{b_{i+1},c_{t+1-(i+1)}}a_{i+1},d_{t+1-(i+1)}$. In terms of this notation,
$$OBRSK(\{\pi_1,\pi_2\})=((\es,\es)\la{b_1,c_t}a_1,d_t)\cdots\la{b_t,c_1}a_t,d_1.$$ 

\begin{lem}\label{l.P(i)_row_strict}
With notation as in the definition of the $OBRSK$ correspondence mentioned above, $P^{(i)}$ is row strict for all $i\in\{1,\ldots,t\}$.
\end{lem}
\begin{Proof}
We will prove the lemma by induction on $i$. The base case (i.e., when $i=1$) of induction is easy to see. 

Now let $i\in\{1,\ldots,t-1\}$. Assume inductively that $P^{(i)}$ is row strict. We will now prove that $P^{(i+1)}$ is row strict. That $P^{(i)}\la{b_{i+1}}a_{i+1}$ is row strict follows in the same way as in \cite{Kr-bkrs}. Note that $P^{(i+1)}$ is obtained from $P^{(i)}\la{b_{i+1}}a_{i+1}$ by adding $d_{t-i}$ at the rightmost end of some row of $P^{(i)}\la{b_{i+1}}a_{i+1}$, say the $k$-th row. It now suffices to ensure that $d_{t-i}$ is strictly bigger than all entries in the $k$-th row of $P^{(i)}\la{b_{i+1}}a_{i+1}$. It follows from the defining properties of the pair of negative skew-symmetric lexicographic arrays $\{\pi_1,\pi_2\}$ that $d_{t-i}$ is bigger than or equal to all entries of $P^{(i)}\la{b_{i+1}}a_{i+1}$. But here we need to prove something sharper, namely: $d_{t-i}$ is strictly bigger than all entries in the $k$-th row of $P^{(i)}\la{b_{i+1}}a_{i+1}$. We will prove this now.

Clearly, all the entries of $P^{(i)}\la{b_{i+1}}a_{i+1}$ are contained in $\{a_1,\ldots,a_{i+1}\}\disjunion\{d_t,\ldots,d_{t+1-i}\}$. Also, it is easy to observe that $a_j<d_{t-i}\ \forall\ j\in\{1,\ldots,i+1\}$. So if the rightmost element of the $k$-th row of $P^{(i)}\la{b_{i+1}}a_{i+1}$ equals $a_j$ for some $j\in\{1,\ldots,i+1\}$, then we are done. Otherwise, the element in the rightmost end of the $k$-th row of $P^{(i)}\la{b_{i+1}}a_{i+1}$ is $d_j$ for some $j\in\{t+1-i,\ldots,t\}$ (say $j_0$). If $d_{j_0}<d_{t-i}$, then we are done. If not, then clearly $d_{j_0}=d_{t-i}$. It then follows from duality that $b_{t+1-j_0}=b_{i+1}$ and it is also clear that $t+1-j_0<i+1$. 

But it is an easy exercise to check that if $l,l'\in\{1,\ldots,t\}$ are such that $l<l'$ and $b_l=b_{l'}$, then the number of the row in which $d_{t+1-l'}$ lies in $P^{(l')}$ is strictly bigger than the number of the row in which $d_{t+1-l}$ lies in $P^{(l')}$ (Here, the row number is counted from top to bottom). So $d_{j_0}$ and $d_{t-i}$ cannot lie in the same row of $P^{(i+1)}$, a contradiction. Hence proved.
\end{Proof}

\bex\label{ex.obrsk}
Let $\pi_1= \left(\begin{array}{ccccc}
17 & 17 & 14 & 10 & 9\\
4 & 3 & 3 & 7 & 4\\ \end{array}\right)\ and\ \pi_2= \left(\begin{array}{ccccc}
25 & 22 & 26 & 26 & 25\\
20 & 19 & 15 & 12 & 12\\ \end{array}\right)$. 
Since two-digit integers are not fit for the Young tableaux package used here for typesetting in latex, we will use some single letter notation for the entries in the above mentioned pair of negative Skew-symmetric lexicographic arrays, and the notation is given as follows:--
$\pi_1= \left(\begin{array}{ccccc}
A & B & C & D & E\\
F & G & H & I & J\\ \end{array}\right)\ and\ \pi_2= \left(\begin{array}{ccccc}
K & L & M & N & O\\
P & Q & R & S & T\\ \end{array}\right)$
where $A=17, B=17, C=14, D=10, E=9, F=4, G=3, H=3, I=7, J=4$ and $K=25, L=22, M=26, N=26, O=25, P=20, Q=19, R=15, S=12, T=12$. 
Then
\[
\begin{array}{l@{\hspace{.9cm}}l}
P^{(0)}=\es
&Q^{(0)}=\es\\ \\
P^{(0)}\la{A}F=\young(F)
&Q^{(0)}\la{dual}O=\young(O)\\ \\
P^{(1)}=\young(FT)
&Q^{(1)}=\young(AO)\\ \\
P^{(1)}\la{B}G=\young(FT)\la{B}G=\young(GT,F) 
&Q^{(1)}\la{dual}N=\young(AO)\la{dual}N=\young(AN,O)\\ \\
P^{(2)}=\young(GT,FS)
&Q^{(2)}=\young(AN,BO)\\ \\
P^{(2)}\la{C}H=\young(GT,FS)\la{C}H=\young(HT,GS,F) 
&Q^{(2)}\la{dual}M=\young(AN,BO)\la{dual}M=\young(AM,BN,O)\\ \\
P^{(3)}=\young(HT,GS,FR)
&Q^{(3)}=\young(AM,BN,CO)\\ \\
P^{(3)}\la{D}I=\young(HT,GS,FR)\la{D}I=\young(HIT,GS,FR) 
&Q^{(3)}\la{dual}L=\young(AM,BN,CO)\la{dual}L=\young(ALM,BN,CO)\\ \\
P^{(4)}=\young(HITQ,GS,FR)
&Q^{(4)}=\young(DALM,BN,CO)\\ \\
P^{(4)}\la{E}J=\young(HITQ,GS,FR)\la{E}J=\young(HJTQ,GIS,FR) 
&Q^{(4)}\la{dual}K=\young(DALM,BN,CO)\la{dual}K=\young(DAKM,BLN,CO)\\ \\ 
\end{array}
\]

\[
\begin{array}{l@{\hspace{.9cm}}l}
P^{(5)}=\young(HJTQ,GISP,FR)
&Q^{(5)}=\young(DAKM,EBLN,CO)\\ \\
\end{array}
\]
\vspace{.5em}

Therefore $OBRSK(\{\pi_1,\pi_2\})=\left(\, \young(HJTQ,GISP,FR)\ ,\
\young(DAKM,EBLN,CO)\right)$.
\eex
\vspace{.5em}
\noindent The proof of the following lemma appears in Section \ref{s.Proofs}.

\begin{lem}\label{l.neg-skewlexarrays-to-neg-skew-bitab}If $\{\pi_1,\pi_2\}$ is a pair of negative Skew-symmetric lexicographic arrays, then $OBRSK(\{\pi_1,\pi_2\})$ is a negative Skew-symmetric notched bitableau. \end{lem}

\begin{lem}\label{l.obrsk-bijection-negative}
The map $OBRSK$ is a degree-preserving bijection from the set of all pairs of negative Skew-symmetric lexicographic arrays to the set of all negative Skew-symmetric notched bitableaux.
\end{lem}
\begin{Proof}
That $OBRSK$ is degree-preserving is obvious.
To show that $OBRSK$ is a bijection, we define its inverse, which we call the \textbf{reverse} of $OBRSK$, or $ROBRSK$. 

Note that the entire procedure used to form
$(P^{(i+1)},Q^{(i+1)})$ from $(P^{(i)},Q^{(i)})$, $a_{i+1}$, $b_{i+1}$, $c_{t-i}$ and $d_{t-i}$, 
$i=1,\ldots,t-1$, is reversible.  In other words, by knowing only
$(P^{(i+1)},Q^{(i+1)})$, we can retrieve $(P^{(i)},Q^{(i)})$,
$a_{i+1}$, $b_{i+1}$, $c_{t-i}$ and $d_{t-i}$. First, we obtain $b_{i+1}$; it is the
minimum entry of $Q^{(i+1)}$. Look at the lowest row in which $b_{i+1}$ appears in $Q^{(i+1)}$, say it is row number $s$ (counting from top to bottom). In the same row (row number $s$, counting from top to bottom) of $P^{(i+1)}$, look at the rightmost entry: this entry is precisely $d_{t-i}$. Remove this entry (which is $d_{t-i}$) from the $s$-th row of $P^{(i+1)}$, that will give us the notched tableau $P^{(i)}\la{b_{i+1}}{a_{i+1}}$. Similarly remove the leftmost entry (which is $b_{i+1}$) from the $s$-th row of $Q^{(i+1)}$ and all other entries in this row of $Q^{(i+1)}$ should be moved one box to the left: this will give us the notched tableau $Q^{(i)}\la{dual}{c_{t-i}}$. 

Then, in the $s$-th row of $P^{(i)}\la{b_{i+1}}{a_{i+1}}$, select the greatest entry which is less than $b_{i+1}$. This entry was the new box of the bounded insertion. If we begin reverse bounded insertion with this entry, we retrieve $P^{(i)}$ and $a_{i+1}$. Look at the \textbf{path} in $P^{(i)}\la{b_{i+1}}{a_{i+1}}$ starting from the $s$-th row to the topmost row, along which this reverse bounded insertion had happened. Trace the \textbf{`dual path'} in $Q^{(i)}\la{dual}{c_{t-i}}$ and do a \textbf{dual} of the reverse bounded insertion (which was done originally on $P^{(i)}\la{b_{i+1}}{a_{i+1}}$ to retrieve $P^{(i)}$ and $a_{i+1}$) on $Q^{(i)}\la{dual}{c_{t-i}}$: that will give us $Q^{(i)}$ and $c_{t-i}$ out of $Q^{(i)}\la{dual}{c_{t-i}}$. 

We call this process of obtaining
$(P^{(i)},Q^{(i)})$, $a_{i+1}$, $b_{i+1}$, $c_{t-i}$ and, $d_{t-i}$ from $(P^{(i+1)},Q^{(i+1)})$
described in the paragraphs above a \textbf{reverse step} and
denote it by $(P^{(i)},Q^{(i)})=(P^{(i+1)},Q^{(i+1)})\ra{b_{i+1},c_{t-i}}a_{i+1},d_{t-i}$. We will call the process of applying all the reverse steps sequentially
to retrieve $\{\pi_1,\pi_2\}$ from $(P^{(t)},Q^{(t)})$ the \textbf{reverse of} $OBRSK$, or $ROBRSK$. 

If $(P^{(t)},Q^{(t)})$ is an arbitrary negative skew-symmetric notched
bitableau (which we do not assume to be $OBRSK(\{\pi_1,\pi_2\})$, for some $\{\pi_1,\pi_2\}$),
then we can still apply a sequence of reverse steps to
$(P^{(t)},Q^{(t)})$, to sequentially obtain $(P^{(i)},Q^{(i)})$,
$a_{i+1}$, $b_{i+1}$, $c_{t-i}$ and, $d_{t-i}$, $i=t-1,\ldots,1$. For this process to be well-defined,
however, it must first be checked that the successive
$(P^{(i)},Q^{(i)})$ are negative skew-symmetric notched bitableaux.
For this, it suffices to prove a statement very similar to that proved in
Lemma \ref{l.neg-skewlexarrays-to-neg-skew-bitab}, namely: `If $(P,Q)$ is a negative
skew-symmetric notched bitableau, then $(P',Q'):=(P,Q)\ra{b,c}a,d$ is a
negative skew-symmetric notched bitableau, $a<b, d<c, a<d, b<c$ are positive
integers, $d$ is greater than or equal to all entries of $P$, and $b$ is less than or equal to all entries of $Q$'.
That $a<b, d<c, a<d, b<c$ are positive integers,  $d$ is greater than or equal to all entries of $P$, and $b$ is less than or equal to
all entries of $Q$ follow immediately from the definition of a
reverse step. That $(P',Q')$ is a negative skew-symmetric notched
bitableau follows in much the same manner as the proof of Lemma
\ref{l.neg-skewlexarrays-to-neg-skew-bitab}; we omit the details. 

It remains to show that the pair of arrays produced by applying this sequence of reverse steps to the
arbitrary skew-symmetric notched bitableau $(P^{(t)},Q^{(t)})$ is skew-symmetric lexicographic. The proof of this uses the duality property of skew-symmetric notched bitableaux, the facts mentioned in the preceding paragraph regarding the integers $a,b,c,d$, and the rest of the proof goes similarly as in the proof of lemma 6.3 of \cite{Kr-bkrs}. 

At each step, $OBRSK$ and the reverse of $ROBRSK$ are inverse to
eachother. Thus they are inverse maps.
\end{Proof}

The map $OBRSK$ can be extended to all pairs of nonvanishing skew-symmetric lexicographic arrays. If $\{\pi_1,\pi_2\}$ is a pair of positive skew-symmetric lexicographic arrays, then define $OBRSK(\{\pi_1,\pi_2\})$ to be $\gi(OBRSK(L(\{\pi_1,\pi_2\})))$. If $\{\pi_1,\pi_2\}$ is a pair of nonvanishing skew-symmetric lexicographic array, with negative and positive parts $\{\pi_1^-,\pi_2^-\}$ and $\{\pi_1^+,\pi_2^+\}$, then define $OBRSK(\{\pi_1,\pi_2\})$ to be the skew-symmetric notched bitableau whose negative and positive parts are $OBRSK(\{\pi_1^-,\pi_2^-\})$ and $OBRSK(\{\pi_1^+,\pi_2^+\})$ (see Figure~\ref{f.obrsk}). 
\setlength{\unitlength}{.50cm}
\begin{figure}[h!]
\begin{picture}(25,8)(0,0)
\put(1,3){$\{\pi_1,\pi_2\}$}
\put(3.5,3){\vector(1,1){2}}
\put(3.5,3){\vector(1,-1){2}}
\put(5.5,0.5){$\{\pi_1^+,\pi_2^+\}$}
\put(5.5,5.5){$\{\pi_1^-,\pi_2^-\}$}
\put(8.5,0.5){\vector(1,0){2}}
\put(8.5,5.5){\vector(1,0){2}}
\put(10.5,0.3){$OBRSK(\{\pi_1^+,\pi_2^+\})$}
\put(10.5,5.3){$OBRSK(\{\pi_1^-,\pi_2^-\})$}
\put(16.8,1){\vector(1,1){2}}
\put(16.8,5){\vector(1,-1){2}}
\put(19,2.8){$OBRSK(\{\pi_1,\pi_2\})$}
\end{picture}
\caption{The map $OBRSK$}
\label{f.obrsk}
\end{figure}
As a consequence of Lemma ~\ref{l.neg-skewlexarrays-to-neg-skew-bitab}, we obtain
\begin{prop}\label{p.obrsk_bijection}
The map $OBRSK$ is a degree-preserving bijection from the set of
all pairs of nonvanishing (resp. negative, positive) skew-symmetric lexicographic arrays to
the set of all nonvanishing (resp. negative, positive) skew-symmetric
notched bitableaux.
\end{prop}

\section{Restricting the OBRSK Correspondence}\label{s.gobrsk_props}

Thus far, there has been no reference to $\ga$, $\gb$, or $\gc$ in
our definition or discussion of the $OBRSK$. In fact, each of
$\ga$, $\gb$, and $\gc$ is used to impose restrictions on the
domain and codomain of the $OBRSK$. It is the $OBRSK$,
with domain and codomain restricted according to $\ga$, $\gb$, and
$\gc$, which is used in Section \ref{s.grobner} to give
geometrical information about $Y\abc$.

In this section, we first show how $\gb$ restricts the domain and
codomain of the $OBRSK$.  We then show how two subsets $T$ and
$W$ of $\NN$, $T$ negative and $W$ positive satisfying condition \ref{e.g.Kreiman's condition (*)}, restrict the domain
and codomain of the $OBRSK$. In Section \ref{s.grobner}, these
two subsets will be replaced by $T_\ga$ and $W_\gc$, subsets of
$\NN$ determined by $\ga$ and $\gc$ respectively.

There is a natural degree-preserving bijection $\psi$ between the set of all pairs of arrays (the arrays in the pair can be so arranged that the first array in the pair is lexicographic and the transpose of the second array in the pair is lexicographic, and this can be done in a unique way) and the set of all pairs of multisets on $\NN$:
\begin{equation}\{\left(\begin{array}{ccc}
b_1 & \cdots & b_{t}\\
a_1 & \cdots & a_{t}\\ \end{array}\right)\ ,\ \left(\begin{array}{ccc}
c_1 & \cdots & c_{t}\\
d_1 & \cdots & d_{t}\\ \end{array}\right)\}\mapsto\ \{\{(a_1,b_1),\ldots,(a_t,b_t)\},\{(d_1,c_1),\ldots,(d_t,c_t)\}\}\end{equation}
We call the image of a pair of skew-symmetric lexicographic arrays under the map $\psi$ to be a \textbf{pair of skew-symmetric multisets on $\NN$}. Recall from \S 4 of \cite{Kr-bkrs} the notion of a non-vanishing (resp. negative, positive) multiset on $\NN$. We call a pair of skew-symmetric multisets on $\NN$ to be \textbf{non-vanishing (resp. negative, positive)} if both the multisets in the pair are non-vanishing (resp. negative, positive). Easy to see that the map $\psi$ restricts to a bijection between pairs of non-vanishing (resp. negative, positive) skew-symmetric lexicographic arrays and pairs of non-vanishing (resp. negative, positive) skew-symmetric multisets on $\NN$. We define the \textbf{degree} of a pair of skew-symmetric multisets on $\NN$ to be the degree of its pre-image under the map $\psi$. For our purposes, it is more convenient to work with pairs of skew-symmetric multisets on $\NN$ than with pairs of skew-symmetric lexicographic arrays.
\begin{cor}\label{c.p.obrsk_bijection}
The map $OBRSK$ induces a degree-preserving bijection from the set of
all pairs of nonvanishing (resp. negative, positive) skew-symmetric multisets on $\NN$ to
the set of all nonvanishing (resp. negative, positive) skew-symmetric
notched bitableaux.
\end{cor}

\subsection*{Restricting by $\gb$}

Let $\gb\in\id$. We say that a \textbf{skew-symmetric notched bitableau} $\mathbf{(P,Q)}$ \textbf{is on} $\mathbf{\gB}$ if all entries of $P$ are in $\ol{\gb}$, all
entries of $Q$ are in $\gb$, and the sum of any entry in $P$ (or in $Q$) with its dual (with respect to $(P,Q)$) is $2d+1$. 

Given any monomial $U$ in $\orootsb$, we can define a monomial $U^{\hash}$ in $\drootsb$ as follows: $U^{\hash}:=\{(c^*,r^*)|(r,c)\in U\}$. We say that a pair $\{V_1,V_2\}$ of skew-symmetric multisets on $\NN$ is a \textbf{pair of skew-symmetric multisets on $\gB$} if $V_1$ is a monomial in $\orootsb$, $V_2$ is a monomial in $\drootsb$, number of elements (counting multiplicities) in $V_1$ and $V_2$ are the same, and $V_2=V_1^{\hash}$. In other words, a general pair of skew-symmetric multisets on $\gB$ will look like $\{V,V^{\hash}\}$ for some monomial $V$ in $\orootsb$. Given any monomial $U$ in $\orootsb$, there is naturally associated to it a pair of skew-symmetric multisets on $\BBZ$ given by $\{U,U^{\hash}\}$. 

It is clear (modulo the observations that any skew-symmetric notched bitableau has to be row-strict by its very definition, and that conditions (iii) and (iv) in the definition of a pair of skew-symmetric lexicographic arrays hold true for the inverse image under the map $\psi$ of any pair of skew-symmetric multisets on $\BBZ$) from the construction of
$OBRSK$ that if $\{U,U^{\hash}\}$ is a pair of nonvanishing skew-symmetric multisets on $\gB$, then
$OBRSK(\{U,U^{\hash}\})$ is a nonvanishing skew-symmetric notched bitableau on
$\gB$, and visa-versa. Thus, as a consequence of Corollary
\ref{c.p.obrsk_bijection}, we obtain

\begin{cor}\label{c.beta_obrsk_bijection}
The map $OBRSK$ restricts to a degree-preserving bijection from the
set of all pairs of nonvanishing (resp. negative, positive) skew-symmetric multisets on $\BBZ$
to the set of all nonvanishing (resp. negative, positive) skew-symmetric notched
bitableaux on $\BBZ$.
\end{cor}

\subsection*{Restricting by $T$ and $W$}
A \textbf{dual pair of chains} in $\NN$ is a pair of subsets $\{C_1=\{(e_1,f_1),\ldots,(e_m,f_m)\},C_2=\{(g_1,h_1),\ldots,(g_m,h_m)\}\}$ such that $C_1$ is a chain in $\NN$ in the sense of section 7 of ~\cite{Kr-bkrs}, and $\{\sigma_1,\sigma_2\}$ is a pair of skew-symmetric lexicographic arrays where 
\begin{equation}\sigma_1= \left(\begin{array}{ccc}
f_1 & \cdots & f_m\\
e_1 & \cdots & e_m\\ \end{array}\right)\ and\ \sigma_2= \left(\begin{array}{ccc}
h_1 & \cdots & h_m\\
g_1 & \cdots & g_m\\ \end{array}\right)\end{equation}.

\bdefn\label{defn.dual-column}
Let $\{U_1,U_2\}$ be a pair of skew-symmetric multisets on $\NN$. Let $\{C_1,C_2\}$ be a dual pair of chains in $\NN$ such that $C_i$ is contained in the underlying set of $U_i$ for all $i=1,2$. Let $\{\pi_{C_1},\pi_{C_2}\}:={\psi}^{-1}(\{C_1,C_2\})$ and $\{\pi_{U_1},\pi_{U_2}\}:={\psi}^{-1}(\{U_1,U_2\})$. Given any column in $\pi_{C_1}$ (say, the $i$-th column counting from left to right), look at the column in $\pi_{U_1}$ having the least possible column number (counting from left to right) which is entrywise the same as the $i$-th column of $\pi_{C_1}$. Call this column of $\pi_{U_1}$ as the $i_{min}$-th column. Let $t$ be the total number of columns in $\pi_{U_1}$. We call the $(t+1-i_{min})$-th column of $\pi_{U_2}$ (counting from left to right) as the \textbf{dual column in} $\pi_{U_2}$ \textbf{corresponding to the} $i$-th \textbf{column of} $\pi_{C_1}$.  
\edefn

\bdefn\label{defn.columns-dual-toeachother}
Given any pair $\{\pi_1,\pi_2\}$ of skew-symmetric lexicographic arrays where 
\begin{equation}\pi_1= \left(\begin{array}{ccc}
b_1 & \cdots & b_t\\
a_1 & \cdots & a_t\\ \end{array}\right)\ and\ \pi_2= \left(\begin{array}{ccc}
c_1 & \cdots & c_t\\
d_1 & \cdots & d_t\\ \end{array}\right)\end{equation},
we say that the column $\left(\begin{array}{c}
b_i\\
a_i\\ \end{array}\right)$ of $\pi_1$ and the column $\left(\begin{array}{c}
c_{t+1-i}\\
d_{t+1-i}\\ \end{array}\right)$ of $\pi_2$ are \textbf{dual to each other w.r.t} $\{\pi_1,\pi_2\}$. Similarly, we say that the column $\left(\begin{array}{c}
b_{t+1-i}\\
a_{t+1-i}\\ \end{array}\right)$ of $\pi_1$ and the column $\left(\begin{array}{c}
c_{i}\\
d_{i}\\ \end{array}\right)$ of $\pi_2$ are \textbf{dual to each other w.r.t} $\{\pi_1,\pi_2\}$.
\edefn

\bdefn\label{defn.dual_pair_of_chains_in_pairs_of_multisets}
Given a pair of $\{U_1,U_2\}$ of skew-symmetric multisets on $\NN$, we say that a dual pair $\{C_1,C_2\}$ of chains in $\NN$ is a \textbf{dual pair of chains in} $\{U_1,U_2\}$, if the following two conditions are satisfied simultaneously:---
\noindent (i) $C_i$ is contained in the underlying set of $U_i$ for all $i=1,2$.
\noindent (ii) If $\{\pi_{C_1},\pi_{C_2}\}:={\psi}^{-1}(\{C_1,C_2\})$ and $\{\pi_{U_1},\pi_{U_2}\}:={\psi}^{-1}(\{U_1,U_2\})$, then given any column of $\pi_{C_1}$ (say, the $i$-th column), the dual column in $\pi_{U_2}$ corresponding to it is entrywise the same as the dual column of the $i$-th column of $\pi_{C_1}$ w.r.t $\{\pi_{C_1},\pi_{C_2}\}$.
\edefn

\bdefn\label{defn.bounded_by_T,W}
Let $T$ and $W$ be negative and positive subsets of $\NN$ respectively satisfying the condition that:---
\begin{equation}\label{e.g.Kreiman's condition (*)}
T_{(1)}, T_{(2)}, W_{(1)}, and\ W_{(2)}\ are\ subsets\ of\ \bN. 
\end{equation}
A nonempty pair of skew-symmetric multisets $\{U_1,U_2\}$ on $\NN$ is said to be \textbf{bounded by} $\mathbf{T, W}$ if for every dual pair $\{C_1,C_2\}$ of chains in $\{U_1,U_2\}$, we have:---
\begin{equation}\label{e.g.ortho_admissible_def}
T\leq {(P_{C_1^-,C_2^-},Q_{C_1^-,C_2^-})}^{up}\ and\ {(P_{C_1^+,C_2^+},Q_{C_1^+,C_2^+})}^{down}\leq W
\end{equation}
(where we use the order on multisets on $\bN^2$ defined in Section 4 of \cite{Kr-bkrs}), and $(P_{C_1^-,C_2^-},Q_{C_1^-,C_2^-})$ (resp. $(P_{C_1^+,C_2^+},Q_{C_1^+,C_2^+})$) is defined to be $OBRSK(\psi^{-1}(\{C_1^-,C_2^-\}))$ (resp. $OBRSK(\psi^{-1}(\{C_1^+,C_2^+\}))$). 

It is worthwhile to note that $(P_{C_1^-,C_2^-},Q_{C_1^-,C_2^-})^{up}$, $(P_{C_1^+,C_2^+},Q_{C_1^+,C_2^+})^{down}$, $T$ and $W$ are subsets of $\NN$, they are NOT pairs of subsets! 

With this definition, the $OBRSK$ correspondence is a bounded
function, in the sense that it maps bounded sets to bounded sets.
More precisely, we have the following Lemma (To understand the statement of this lemma, we need to recall the notion of a semistandard notched bitableau being bounded by $\mathbf{T,W}$ from \S 5 of \cite{Kr-bkrs}), whose proof appears
in Section \ref{s.Proofs}.
\edefn

\begin{lem}\label{l.g.dominance}
If a pair $\{U_1,U_2\}$ of nonvanishing skew-symmetric multisets on $\NN$ is bounded by $T,W$, then
$OBRSK(\{U_1,U_2\})$ is bounded by $T,W$. [Note that here, by $OBRSK(\{U_1,U_2\})$, we mean $OBRSK({\psi}^{-1}(\{U_1,U_2\}))$.]
\end{lem}

\bdefn\label{defn.dual_pair_of_chains_in_BBZ}
A dual pair $\{C_1,C_2\}$ of chains in $\NN$ is called a \textbf{dual pair of chains in $\BBZ$} if $\{C_1,C_2\}$ is a pair of skew-symmetric multisets on $\BBZ$. Clearly then, a general dual pair of chains in $\BBZ$ will look like $\{C,C^{\hash}\}$ for some extended $\gb$-chain $C$ in $\orootsb$.
\edefn

\bdefn\label{defn.subsets-for-bitableau}
Given any row-strict notched bitableau $(P,Q)$, we associate to it $2$ subsets of $\NN$ as follows:---
Let $P_1$ and $Q_1$ denote the topmost row of $P$ and $Q$ respectively. Let $p_{11}<\ldots<p_{1k_1}$ and $q_{11}<\ldots<q_{1k_1}$ denote the entries of $P_1$ and $Q_1$ respectively. We denote by $(P,Q)^{up}$ the subset of $\NN$ given by $\{(p_{11},q_{11}),\ldots,(p_{1k_1},q_{1k_1})\}$. We denote by $(P,Q)^{down}$ the subset of $\NN$ obtained similarly if we work with the lower-most rows of $P$ and $Q$, instead of the topmost rows. We call $(p_{1j},q_{1j})$ the \textbf{j-th element of} $(P,Q)^{up}$ and, we denote by $(P,Q)^{up}_{\leq j}$ the subset $\{(p_{11},q_{11}),\ldots,(p_{1j},q_{1j})\}$ of $(P,Q)^{up}$.
\edefn 

\begin{rem}\label{r.a_particular_case_of_dual_pair_of_chains}
Note that if $\{U_1,U_2\}$ is a pair of skew-symmetric multisets on $\BBZ$, i.e., if $U_2=U_1^{\hash}$, then any dual pair of chains $\{C_1,C_2\}$ in $\{U_1,U_2\}$ must be a dual pair of chains in $\BBZ$, in other words, we must have $C_2=C_1^{\hash}$ where $C_1$ is an extended $\gb$-chain in $\orootsb$.
\end{rem}

\begin{rem}\label{r.a_particular_case_of_bounded_by_TW}
Let $T$ and $W$ be negative and positive subsets of $\NN$ respectively satisfying \ref{e.g.Kreiman's condition (*)}. A nonempty pair of skew-symmetric multisets $\{U,U^{\hash}\}$ on $\BBZ$ is said to be \textbf{bounded by} $\mathbf{T, W}$ if for every dual pair of chains $\{C,C^{\hash}\}$ in $\BBZ$ which is contained in the underlying set of $\{U,U^{\hash}\}$,
\begin{equation}\label{e.g.ortho_admissible_def}
T\leq (P_{C^-,{C^-}^{\hash}},Q_{C^-,{C^-}^{\hash}})^{up}\ and\ (P_{C^+,{C^+}^{\hash}},Q_{C^+,{C^+}^{\hash}})^{down}\leq W
\end{equation}
(where we use the order on multisets on $\bN^2$ defined in Section 4 of \cite{Kr-bkrs}), and $(P_{C^-,{C^-}^{\hash}},Q_{C^-,{C^-}^{\hash}})$ (resp. $(P_{C^+,{C^+}^{\hash}},Q_{C^+,{C^+}^{\hash}})$) is defined to be $OBRSK(\psi^{-1}(\{C^-,{C^-}^{\hash}\}))$ (resp. $OBRSK(\psi^{-1}(\{C^+,{C^+}^{\hash}\}))$). 

It is worthwhile to note that $(P_{C^-,{C^-}^{\hash}},Q_{C^-,{C^-}^{\hash}})^{up}$, $(P_{C^+,{C^+}^{\hash}},Q_{C^+,{C^+}^{\hash}})^{down}$, $T$ and $W$ are subsets of $\NN$, they are NOT pairs of subsets! 
\end{rem}

Let $T$ and $W$ be negative and positive subsets of $\gB$,
respectively satisfying \ref{e.g.Kreiman's condition (*)}. Combining Corollary \ref{c.beta_obrsk_bijection} and
Lemma \ref{l.g.dominance}, we obtain
\begin{cor}\label{c.g.beta_dominance_2}
For any positive integer $m$, the number of pairs of 
nonvanishing skew-symmetric multisets on $\BBZ$ bounded by $T, W$ of degree $2m$ is less than or
equal to the number of nonvanishing skew-symmetric notched bitableaux on $\BBZ$ bounded by $T, W$ of degree $2m$.
\end{cor}

\section{The initial ideal}\label{s.grobner}
Let $P=\field[X_{(r,c)}\st (r,c)\in\orootsb]$. Recall the concept of a \textit{Pfaffian} (denoted by $f_{\theta,\gb}$ for $\theta\in\id$) from \S~\ref{ss.ideal-of-tgtcone} of this paper. We call $f=f_{\gt_1,\gb}\cdots f_{\gt_r,\gb}\in P$ a
\textbf{standard monomial} if $\gt_1,\ldots,\gt_r\in\id$, 
\begin{equation}\label{e.r.stand_mon_1}
\gt_1\leqT \cdots\leqT \gt_r
\end{equation}
and for each $i\in\{1,\ldots,r\}$, either
\begin{equation}\label{e.r.stand_mon_2}
\gt_i\ltT \gb\ \ \hbox{ or }\ \ \gt_i\gtT \gb.
\end{equation}
If in addition, for $\ga,\gc\in\id$,
\begin{equation}\label{e.r.stand_mon_on_ab} \ga\leq \gt_1\quad\text{and} \quad
\gt_r\leq\gc,
\end{equation}
then we say that $f$ is \textbf{standard on} $Y\abc$. We define the \textbf{degree} of the standard monomial $f_{\gt_1,\gb}\cdots f_{\gt_r,\gb}$ to be the sum of the $\gb$-degrees of $\gt_1,\ldots,\gt_r$ where for any $\gt\in\id$, the $\gb$-degree is defined to be one-half the cardinality of $\gt\setminus\gb$. 

We remark that, in general, a standard monomial is not a monomial
in the affine coordinates $X_{(r,c)},(r,c)\in\orootsb$; rather,
it is a polynomial. It is only a monomial in the $f_{\gt,\gb}$'s. The following result follows in exactly the same way as in the proof of Proposition~3.2.1 of ~\cite{sthesis}:---
\begin{thm}\label{t.r.stand_mons_form_basis}
The standard monomials on $Y_{\ga,\gb}^\gc$ form a basis for
$\field[Y_{\ga,\gb}^\gc]$.
\end{thm}

We wish to give a different indexing set for the standard monomials on $Y\abc$. 
Let $I_{\gb}(Skew-symm)$ denote the set of all pairs $(\pr,\qr)$ such that all of the following conditions are satisfied:---
\begin{itemize}
\item $\pr\subset\ol{\gb}$.
\item $\qr\subset\gb$.
\item $|\pr|=|\qr|$ and this cardinality is \textit{even}.
\item If $\pr=\{r_1<\cdots<r_{2l}\}$ and $\qr=\{s_1<\cdots<s_{2l}\}$, then $r_i+s_{2l+1-i}=2d+1\ \forall\ i\in\{1,\ldots,2l\}$.
\end{itemize}

Defining $\pr-\qr:=\pr\disjunion(\gb\setminus
\qr)$ (see Section 4 of \cite{Kr-bkrs}), we have the following fact,
which is easily verified:
\begin{equation*}
\hbox{The map }(\pr,\qr)\mapsto \pr-\qr\hbox{ is a bijection from
}I_{\gb}(Skew-symm)\hbox{ to }\id,
\end{equation*}
(Indeed, the inverse map is given by $\gt\mapsto
(\gt\setminus\gb,\gb\setminus\gt)$).

Note that under this bijection, $(\es,\es)$ maps to $\gb$. Let
$(\pr_\ga,\qr_\ga)$ and $(\pr_\gc,\qr_\gc)$ be the preimages of
the elements $\ga$ and $\gc$ (of $\id$) respectively.  Define $T_\ga$ and $W_\gc$ to be
any subsets of $\BBZ$ such that $(T_\ga)_{(1)}=\pr_\ga$,
$(T_\ga)_{(2)}=\qr_\ga$, $(W_\gc)_{(1)}=\pr_\gc$,
$(W_\gc)_{(2)}=\qr_\gc$. Observe that $T_\ga$ and $W_\gc$ satisfy \ref{e.g.Kreiman's condition (*)}.

Under this identification of $I_{\gb}(Skew-symm)$ with $\id$, the inequalities which define non-vanishing skew-symmetric notched bitableaux on $\BBZ$ bounded by $T_\ga$, $W_\gc$ (these are inequalities (3),(4),(5) of \cite{Kr-bkrs}) are precisely the inequalities which define the standard monomials on $Y_{\ga,\gb}^\gc$ (the inequalities (\ref{e.r.stand_mon_1}), (\ref{e.r.stand_mon_2}), (\ref{e.r.stand_mon_on_ab}) of this paper). Thus we obtain
\begin{lem}\label{l.r.skewsymm_index}
The degree $2m$ nonvanishing skew-symmetric notched bitableaux on
$\BBZ$ bounded by $T_\ga, W_\gc$ form an indexing set for the
degree $m$ standard monomials on $Y_{\ga,\gb}^\gc$.
\end{lem}

Recall that $Chains_\ga^\gc(\gb)$ is the set $\{X_C\st C\ is\ a\ non-vanishing\ extended\ \gb-chain\ in\ \orootsb\ such\ that\ either\ (i)\ or\ (ii)\ of\ \ref{eq.chains_again}\ holds\}$.
\begin{equation}\label{eq.chains_again}
(i) C^-\ is\ non-empty\ and\ \ga\not\leq w_{C^-}^-(\gb).
(ii) C^+\ is\ non-empty\ and\ w_{C^+}^+(\gb)\not\leq\gc.
\end{equation}

For the rest of this section, we will use extensively the terminology and notation of \S 4 of \cite{Kr-bkrs}.

\begin{rem}\label{r.translate-cond-eq.chains}
Let $C$ be a non-vanishing extended $\gb$-chain in $\orootsb$. Note that $w_{C^-}^-(\gb)=\id(\diagb_{C^-}(-))$ and $w_{C^+}^+(\gb)=\id(\diagb_{C^+}(+))$. Consider the dual pairs $\{C^-,{C^-}^{\hash}\}$ and $\{C^+,{C^+}^{\hash}\}$ of chains in $\BBZ$, as defined above. Consider the row-strict notched bitableau $(P_{C^-,{C^-}^{\hash}},Q_{C^-,{C^-}^{\hash}})$ which is by definition $OBRSK(\psi^{-1}(C^-,{C^-}^{\hash}))$, and similarly consider the row-strict notched bitableau $(P_{C^+,{C^+}^{\hash}},Q_{C^+,{C^+}^{\hash}})$. Then consider the subsets of $\NN$ given by $(P_{C^-,{C^-}^{\hash}},Q_{C^-,{C^-}^{\hash}})^{up}$ and $(P_{C^+,{C^+}^{\hash}},Q_{C^+,{C^+}^{\hash}})^{down}$ (see definition ~\ref{defn.subsets-for-bitableau}). It is easy to see that since $C$ be an extended $\gb$-chain in $\orootsb$, the subsets $(P_{C^-,{C^-}^{\hash}},Q_{C^-,{C^-}^{\hash}})^{up}$ and $(P_{C^+,{C^+}^{\hash}},Q_{C^+,{C^+}^{\hash}})^{down}$ of $\NN$ are actually subsets of $\BBZ$. 

It is easy to observe that 
\begin{itemize}
\item ${\diagb_{C^-}(-)}_{(1)}=(P_{C^-,{C^-}^{\hash}},Q_{C^-,{C^-}^{\hash}})^{up}_{(1)}$ 
\item ${\diagb_{C^-}(-)}_{(2)}=(P_{C^-,{C^-}^{\hash}},Q_{C^-,{C^-}^{\hash}})^{up}_{(2)}$ 
\item ${\diagb_{C^+}(+)}_{(1)}=(P_{C^+,{C^+}^{\hash}},Q_{C^+,{C^+}^{\hash}})^{down}_{(1)}$ and 
\item ${\diagb_{C^+}(+)}_{(2)}=(P_{C^+,{C^+}^{\hash}},Q_{C^+,{C^+}^{\hash}})^{down}_{(2)}$
\end{itemize}
where for any multiset $U=\{(e_1,f_1),(e_2,f_2),\ldots\}$ on $\bN^2$, $U_{(1)}$ and $U_{(2)}$ are defined to be the multisets
$\{e_1,e_2,\ldots\}$ and $\{f_1,f_2,\ldots\}$ respectively on $\bN$ as in \S 4 of \cite{Kr-bkrs}.

It is now easy to see that the conditions $(i)$ and $(ii)$ for the non-vanishing extended $\gb$-chain $C$ in $\orootsb$ as mentioned in equation \ref{eq.chains_again} above can be translated into the conditions $(i)'$ and $(ii)'$ as mentioned below:---
\begin{equation}\label{eq.chains_translated_1}
(i)'\ C^-\ is\ non-empty\ and\ R_{\ga}-S_{\ga}\not\leq (P_{C^-,{C^-}^{\hash}},Q_{C^-,{C^-}^{\hash}})^{up}_{(1)}-(P_{C^-,{C^-}^{\hash}},Q_{C^-,{C^-}^{\hash}})^{up}_{(2)}.
\end{equation}

\begin{equation}\label{eq.chains_translated_2}
(ii)'\ C^+\ is\ non-empty\ and\ (P_{C^+,{C^+}^{\hash}},Q_{C^+,{C^+}^{\hash}})^{down}_{(1)}-(P_{C^+,{C^+}^{\hash}},Q_{C^+,{C^+}^{\hash}})^{down}_{(2)}\not\leq R_{\gc}-S_{\gc}.
\end{equation}
\end{rem}

\begin{lem}\label{l.r.admiss_index}
The pairs of non-vanishing skew-symmetric multisets on $\BBZ$ bounded by $T_\ga, W_\gc$ of degree $2m$ form
an indexing set for the degree $m$ monomials of $P/\langle Chains_\ga^\gc(\gb)\rangle$.
\end{lem}
\begin{Proof}
Note that 
\begin{align*}
\langle Chains_\ga^\gc(\gb)\rangle&= \langle x_C\mid C\hbox{ an extended $\gb$ chain in $\orootsb$}, \hbox{ either }\ref{eq.chains_translated_1}\hbox{ or } \ref{eq.chains_translated_2}\hbox{ holds }\rangle\\
&=\langle x_C\mid C\hbox{ an extended $\gb$ chain in $\orootsb$}, T_\ga\not\leq (P_{C^-,{C^-}^{\hash}},Q_{C^-,{C^-}^{\hash}})^{up}\\
&\ \hbox{ or }(P_{C^+,{C^+}^{\hash}},Q_{C^+,{C^+}^{\hash}})^{down}\not\leq W_\gc\rangle.
\end{align*}
Therefore,\\

\noindent $x_U$ is a monomial in $P/\langle Chains_\ga^\gc(\gb)\rangle$
\begin{itemize}
\item[$\iff$] $x_U$ is not divisible by any $x_C$, $C$ an extended $\gb$ chain in $\orootsb$ 
such that $T_\ga\not\leq (P_{C^-,{C^-}^{\hash}},Q_{C^-,{C^-}^{\hash}})^{up}$ or $(P_{C^+,{C^+}^{\hash}},Q_{C^+,{C^+}^{\hash}})^{down}\not\leq W_\gc$

\item[$\iff$] $U$ contains no extended $\gb$-chains $C$ such that $T_\ga\not\leq
(P_{C^-,{C^-}^{\hash}},Q_{C^-,{C^-}^{\hash}})^{up}$ or $(P_{C^+,{C^+}^{\hash}},Q_{C^+,{C^+}^{\hash}})^{down}\not\leq W_\gc$

\item[$\iff$] $T_\ga\leq (P_{C^-,{C^-}^{\hash}},Q_{C^-,{C^-}^{\hash}})^{up}$ and $(P_{C^+,{C^+}^{\hash}},Q_{C^+,{C^+}^{\hash}})^{down}\leq W_\gc$, for every extended $\gb$-chain $C$ in $U$

\item[$\iff$] The pair $\{U,U^{\hash}\}$ of skew-symmetric multisets on $\BBZ$ is bounded by $T_\ga,W_\gc$.
\end{itemize}
\end{Proof}
\noindent We are now ready to prove the main result of the paper.
\begin{proof}[Proof of Theorem \ref{t.main}.] We wish to show that $\init I=\langle Chains_\ga^\gc(\gb)\rangle$. Since we already know from Remark \ref{r.Chains-in-initI} that $Chains_\ga^\gc(\gb)\subseteq\init I$, it follows that $\langle Chains_\ga^\gc(\gb)\rangle\subseteq\init I$. For any $m\geq 1$,\\

\noindent\# of degree $m$ monomials in $P/\langle Chains_\ga^\gc(\gb)\rangle$
\begin{itemize}
\item[$\stackrel{a}{=}$] \# of pairs of non-vanishing skew-symmetric multisets on $\BBZ$ bounded by $T_\ga, W_\gc$ of degree $2m$

\item[$\stackrel{b}{\leq}$] \# of nonvanishing skew-symmetric notched bitableaux on $\BBZ$ bounded by $T_\ga, W_\gc$ of degree $2m$

\item[$\stackrel{c}{=}$] \# of degree $m$ standard monomials on
$Y_{\ga,\gb}^\gc$

\item[$\stackrel{d}{=}$] \# of degree $m$ monomials in
$P/\init I$,
\end{itemize}

\noindent where $a$ follows from Lemma \ref{l.r.admiss_index}, $b$
from Corollary \ref{c.g.beta_dominance_2}, $c$ from Lemma
\ref{l.r.skewsymm_index}, and $d$ from the fact that standard
monomials on $Y_{\ga,\gb}^\gc$ and the monomials in
$P/\init I$ both induce
homogeneous bases for $P/I$. Thus $\langle Chains_\ga^\gc(\gb)\rangle\supseteq\init I$.

We point out that, as a consequence of this proof, inequality b is
actually an equality.
\end{proof}

\section{Proofs}\label{s.Proofs}
In this section, we will use extensively the terminology and notation of \S 4 of \cite{Kr-bkrs}.

\subsection{Proof of Lemma~\ref{l.neg-skewlexarrays-to-neg-skew-bitab}}\label{ss.l.neg-skewlexarrays-to-neg-skew-bitab}
\begin{Proof} 
The proof is by induction, the base case of induction is easy to see. Let $\{\pi_1^{(t-1)},\pi_2^{(t-1)}\}$ be a pair of negative Skew-symmetric lexicographic arrays given by:---
\begin{equation}\pi_1^{(t-1)}= \left(\begin{array}{ccc}
b_1 & \cdots & b_{t-1}\\
a_1 & \cdots & a_{t-1}\\ \end{array}\right)\ and\ \pi_2^{(t-1)}= \left(\begin{array}{ccc}
c_2 & \cdots & c_{t}\\
d_2 & \cdots & d_{t}\\ \end{array}\right)\end{equation} 
Let $(P,Q)=OBRSK(\{\pi_1^{(t-1)},\pi_2^{(t-1)}\})$. Assume inductively that $(P,Q)$ is a negative Skew-symmetric notched bitableau. Now let $\{\pi_1^{(t)},\pi_2^{(t)}\}$ be a pair of negative Skew-symmetric lexicographic arrays given by:---
\begin{equation}\pi_1^{(t)}= \left(\begin{array}{ccc}
b_1 & \cdots & b_{t}\\
a_1 & \cdots & a_{t}\\ \end{array}\right)\ and\ \pi_2^{(t)}= \left(\begin{array}{ccc}
c_1 & \cdots & c_{t}\\
d_1 & \cdots & d_{t}\\ \end{array}\right)\end{equation}   
that is,  $\{\pi_1^{(t)},\pi_2^{(t)}\}$ is obtained by attaching the elements $b_t,a_t,c_1,d_1$ to $\{\pi_1^{(t-1)},\pi_2^{(t-1)}\}$ in a way such that the resulting pair of arrays $\{\pi_1^{(t)},\pi_2^{(t)}\}$ is again negative Skew-symmetric lexicographic. Let $(P',Q')=OBRSK(\{\pi_1^{(t)},\pi_2^{(t)}\})$. It suffices to show that $(P',Q')$ is also a negative Skew-symmetric notched bitableau. 

We will first prove that $(P',Q')$ is a negative semistandard notched bitableau. The fact that $P'$ is row-strict follows from lemma ~\ref{l.P(i)_row_strict}. It then follows from duality that $Q'$ is also row-strict. Hence $(P',Q')$ is row-strict. Let $r'$ be the total number of rows of $P'(or\ Q')$. Let $P_i'(resp.\ Q_i')$ denote the set of all elements in the $i$-th row of $P'(resp.\ Q')$. It needs to be shown that $P_i'-Q_i'\leq P_{i+1}'-Q_{i+1}'$ for $1\leq i\leq r'-1$, and $P_{r'}'-Q_{r'}'<\es$. We will prove the former statement first. To prove that $P_i'-Q_i'\leq P_{i+1}'-Q_{i+1}'$ for $1\leq i\leq r'-1$, there will be three non-trivial possibilities which will be registered below as cases I,II, and III respectively. Let $P_i(resp\ Q_i)$ denote the set of all elements in the $i$-th row of $P(resp\ Q)$.

\noindent\textbf{Case I}: $a_t$ and $d_1$ are added to the first row of $P$ and $c_1$ and $b_t$ are added to the first row of $Q$. All rows other than the $1$-st row of $P$ and $Q$ remain unchanged. 

In this case, it suffices to show that $P_1'-Q_1'\leq P_{2}'-Q_{2}'$ which is the same as showing that $P_1'\disjunion Q_2'\leq P_2'\disjunion Q_1'$. Note that $P_1'=P_1\disjunion\{a_t,d_1\}$, $Q_1'=Q_1\disjunion\{c_1,b_t\}$, $P_2'=P_2$ and $Q_2'=Q_2$. Since $(P,Q)$ is assumed to be a negative Skew-symmetric notched bitableau, we have $P_1\disjunion Q_2\leq P_2\disjunion Q_1$. It then suffices to show that $P_1\disjunion Q_2\disjunion\{a_t,d_1\}\leq P_2\disjunion Q_1\disjunion\{c_1,b_t\}$. 

Since $a_t$ and $d_1$ are the duals (with respect to the pair $\{\pi_1^{(t)},\pi_2^{(t)}\}$ of arrays) of $c_1$ and $b_t$ respectively, and the elements of $P_1\disjunion Q_2$ are the duals (with respect to $(P,Q)$) of the elements of $P_2\disjunion Q_1$, therefore the positions at which $c_1$ and $b_t$ appear in the set $P_2\disjunion Q_1\disjunion\{c_1,b_t\}$ (when the elements of the set are written in ascending order) are `dual' to the positions at which $a_t$ and $d_1$ appear in the set $P_1\disjunion Q_2\disjunion\{a_t,d_1\}$ (when the elements of the set are written in ascending order). \textit{[In the sense that if $c_1$ appears at the $l$-th position (counting from left to right) in the set $P_2\disjunion Q_1\disjunion\{c_1,b_t\}$ (when the elements of the set are written in ascending order), then $a_t$ appears at the reverse $l$-th position (that is, the $l$-th position counting from right to left) in the set $P_1\disjunion Q_2\disjunion\{a_t,d_1\}$ (when the elements of the set are written in ascending order), and a similar thing is true for the positions of $b_t$ and $d_1$.]} The above fact, together with the facts that $a_t<d_1$, $b_t<c_1$, $a_t\leq b_t$ (in fact, $a_t<b_t$) and $d_1\leq c_1$ (in fact, $d_1<c_1$) imply easily that $P_1\disjunion Q_2\disjunion\{a_t,d_1\}\leq P_2\disjunion Q_1\disjunion\{c_1,b_t\}$. Hence we are done in this case. 

\noindent\textbf{Case II}: $x_p$ bumps $y_p$ from $P_i$ and $y_p$ bumps $z_p$ from $P_{i+1}$, and the dual bumping happens on $Q_i$ and $Q_{i+1}$. Let us express the dual bumping by saying that $x_q$ bumps $y_q$ from $Q_i$ and $y_q$ bumps $z_q$ from $Q_{i+1}$.

Clearly then, $x_p\leq y_p\leq z_p<b_t$ and hence $x_q\geq y_q\geq z_q$ by the property (ii) in the definition of a Skew-symmetric notched bitableau. Again since all entries in $Q$ are $\geq b_t$, we get that $x_q\geq y_q\geq z_q\geq b_t$. It is also easy to note that
\begin{center}
$P_i'=(P_i\setminus\{y_p\})\disjunion\{x_p\}$\hspace{3cm} $P_{i+1}'=(P_{i+1}\setminus\{z_p\})\disjunion\{y_p\}$\\
$Q_i'=(Q_i\setminus\{y_q\})\disjunion\{x_q\}$\hspace{3cm} $Q_{i+1}'=(Q_{i+1}\setminus\{z_q\})\disjunion\{y_q\}$
\end{center}
We need to show that $P_i'-Q_i'\leq P_{i+1}'-Q_{i+1}'$, or in other words, $P_i'\disjunion Q_{i+1}'\leq P_{i+1}'\disjunion Q_i'$. Note that $P_i'\disjunion Q_{i+1}'=[(P_i\disjunion Q_{i+1})\setminus\{y_p,z_q\}]\disjunion\{x_p,y_q\}$ and $P_{i+1}'\disjunion Q_i'=[(P_{i+1}\disjunion Q_i)\setminus\{z_p,y_q\}]\disjunion\{x_q,y_p\}$. It suffices to prove that $(P_i\disjunion Q_{i+1})\setminus\{y_p,z_q\}\leq (P_{i+1}\disjunion Q_i)\setminus\{z_p,y_q\}$ because if we can prove this much, then since $x_p\leq y_p, x_q\geq y_q, x_p<y_q\ and\ y_p<x_q$, we can give an argument exactly similar to Case I to prove that $[(P_i\disjunion Q_{i+1})\setminus\{y_p,z_q\}]\disjunion\{x_p,y_q\}\leq [(P_{i+1}\disjunion Q_i)\setminus\{z_p,y_q\}]\disjunion\{x_q,y_p\}$. 

We will now prove that $(P_i\disjunion Q_{i+1})\setminus\{y_p,z_q\}\leq (P_{i+1}\disjunion Q_i)\setminus\{z_p,y_q\}$. The proof of this follows easily from the facts mentioned below: Since $z_p<b_t$ and all entries of $Q_i$ are $\geq b_t$, therefore $z_p$ is the smallest element in $P_{i+1}\disjunion Q_i$ which is $\geq y_p$. It then follows from duality that $z_q$ is the biggest element in $P_i\disjunion Q_{i+1}$ that is $\leq y_q$. Also since $y_p<b_t\leq z_q$, we have $y_p<z_q$, and hence by duality $y_q>z_p$. Since $(P,Q)$ is assumed to be a negative Skew-symmetric notched bitableau, we have $P_i\disjunion Q_{i+1}\leq P_{i+1}\disjunion Q_i$. All these facts put together prove the required thing easily and we are done in this case.

\noindent\textbf{Case III}: $x_p$ bumps $y_p$ from $P_i$ and, $y_p$ along with $d_1$ are added to $P_{i+1}$. The dual phenomenon happens with $Q_i$ and $Q_{i+1}$, we express the dual phenomenon by saying that $x_q$ bumps $y_q$ from $Q_i$ and, $y_q$ along with $b_t$ are added to $Q_{i+1}$.

Clearly then, $x_p\leq y_p<b_t$ and hence $x_q\geq y_q$ by the property (ii) in the definition of a Skew-symmetric notched bitableau. Again since all entries in $Q$ are $\geq b_t$, we get that $x_q\geq y_q\geq b_t$. It is also easy to note that 
\begin{center}
$P_i'=(P_i\setminus\{y_p\})\disjunion\{x_p\}$\hspace{3cm} $P_{i+1}'=P_{i+1}\disjunion\{y_p\}\disjunion\{d_1\}$\\
$Q_i'=(Q_i\setminus\{y_q\})\disjunion\{x_q\}$\hspace{3cm} $Q_{i+1}'=Q_{i+1}\disjunion\{b_t\}\disjunion\{y_q\}$
\end{center}
We need to show that $P_i'-Q_i'\leq P_{i+1}'-Q_{i+1}'$, or in other words, $P_i'\disjunion Q_{i+1}'\leq P_{i+1}'\disjunion Q_i'$. Note that $P_i'\disjunion Q_{i+1}'=(P_i\disjunion Q_{i+1}\setminus\{y_p\})\disjunion\{b_t\}\disjunion\{x_p,y_q\}$ and $P_{i+1}'\disjunion Q_i'=(P_{i+1}\disjunion Q_i\setminus\{y_q\})\disjunion\{d_1\}\disjunion\{x_q,y_p\}$. Using arguments similar to Case II, it follows that it is enough to prove that $(P_i\disjunion Q_{i+1}\setminus\{y_p\})\disjunion\{b_t\}\leq (P_{i+1}\disjunion Q_i\setminus\{y_q\})\disjunion\{d_1\}$, which we will do now.

Note that there are no elements of $P_{i+1}$ which are $\geq y_p$ and $<b_t$. Also since $b_t$ is $\leq$ all elements of $Q_i$, therefore we can conclude that there are no elements of $P_{i+1}\disjunion Q_i$ which are $\geq y_p$ and $<b_t$. Let $\alpha_1\leq\ldots\leq\alpha_k$ and $\beta_1\leq\ldots\leq\beta_k$ denote the multisets $P_i\disjunion Q_{i+1}$ and $P_{i+1}\disjunion Q_i$ respectively. Since $(P,Q)$ is assumed to be a negative Skew-symmetric notched bitableau, we have $P_i\disjunion Q_{i+1}\leq P_{i+1}\disjunion Q_i$, that is, $\alpha_i\leq\beta_i\ \forall i\in\{1,\ldots,k\}$. Let $y_p=\alpha_{l+1}$. Since there are no elements of $P_{i+1}\disjunion Q_i$ which are $\geq y_p$ and $<b_t$, it follows that $\beta_{l+1}\geq b_t$ and $\beta_l$ should either be $<y_p$ or $\geq b_t$. Since $y_p<b_t$, it is now easy to see that $(P_i\disjunion Q_{i+1}\setminus\{y_p\})\disjunion\{b_t\}\leq P_{i+1}\disjunion Q_i$. 

Since the elements of $P_i\disjunion Q_{i+1}$ and $P_{i+1}\disjunion Q_i$ are dual to each other with respect to the natural partial order $\leq$ on the set of all integers, and the total number of elements in $P_i\disjunion Q_{i+1}(or\ in\ P_{i+1}\disjunion Q_i)$ is even, it follows that $\beta_{l+1}\neq y_q$. So either $y_q=\beta_j$ for some $j\in\{1,\ldots,l\}$ or $y_q=\beta_j$ for some $j\in\{l+2,\ldots,k\}$. Let us write the elements of $(P_i\disjunion Q_{i+1}\setminus\{y_p\})\disjunion\{b_t\}$ as $\gamma_1\leq\ldots\leq\gamma_k$. Clearly then $b_t=\gamma_s$ for some $s\geq l+1$. Since $y_p<b_t$, it follows from duality that $y_q>d_1$. It also follows from duality that there are no elements of $P_i\disjunion Q_{i+1}$ which are $\leq y_q$ and $>d_1$. Hence there are no elements of $(P_i\disjunion Q_{i+1})\setminus\{y_p\}$ which are $\leq y_q$ and $>d_1$. Since $y_q\geq b_t$, it also follows that EITHER there are no elements of $((P_i\disjunion Q_{i+1})\setminus\{y_p\})\disjunion\{b_t\}$ which are $\leq y_q$ and $>d_1$ OR the only possible element in $((P_i\disjunion Q_{i+1})\setminus\{y_p\})\disjunion\{b_t\}$ which is $\leq y_q$ and $>d_1$ is $b_t$ itself (in which case $b_t>d_1$). 

Now suppose that $y_q=\beta_j$ for some $j\in\{1,\ldots,l\}$. Then since $y_p=\alpha_{l+1}, y_p<b_t$ and $b_t=\gamma_s$ for some $s\geq l+1$, it follows that $\gamma_j$ should be $\leq d_1$ and $\gamma_j\neq b_t$. Now since we already know that $(P_i\disjunion Q_{i+1}\setminus\{y_p\})\disjunion\{b_t\}\leq P_{i+1}\disjunion Q_i$, it is easy to see that $(P_i\disjunion Q_{i+1}\setminus\{y_p\})\disjunion\{b_t\}\leq (P_{i+1}\disjunion Q_i\setminus\{y_q\})\disjunion\{d_1\}$ in the case when $y_q=\beta_j$ for some $j\in\{1,\ldots,l\}$. Let us now work out the other case, that is, the case when $y_q=\beta_j$ for some $j\in\{l+2,\ldots,k\}$. Then there are two possibilities for $\gamma_j$: EITHER $\gamma_j\leq d_1$ OR $\gamma_j=b_t$ (where $b_t>d_1$) and $\gamma_t\leq d_1$ for all $t\in\{1,\ldots,j-1\}$. If $\gamma_j\leq d_1$, then since we already know that $(P_i\disjunion Q_{i+1}\setminus\{y_p\})\disjunion\{b_t\}\leq P_{i+1}\disjunion Q_i$, it is easy to see that $(P_i\disjunion Q_{i+1}\setminus\{y_p\})\disjunion\{b_t\}\leq (P_{i+1}\disjunion Q_i\setminus\{y_q\})\disjunion\{d_1\}$. But if $\gamma_j=b_t$ (where $b_t>d_1$) and $\gamma_t\leq d_1$ for all $t\in\{1,\ldots,j-1\}$, then since $\beta_{l+1}\geq b_t$, $b_t>d_1$ and $(P_i\disjunion Q_{i+1}\setminus\{y_p\})\disjunion\{b_t\}\leq P_{i+1}\disjunion Q_i$, it follows easily that $(P_i\disjunion Q_{i+1}\setminus\{y_p\})\disjunion\{b_t\}\leq (P_{i+1}\disjunion Q_i\setminus\{y_q\})\disjunion\{d_1\}$. Hence we are done in Case III. 

We have now proved that $(P',Q')$ is a semistandard notched bitableau. To prove that the semistandard notched bitableau $(P',Q')$ is negative, it is enough to prove that $P_{r'}'-Q_{r'}'<\es$ where $r'$ denotes the total number of rows of $P'(or\ Q')$. If $P_{r'}'=P_{r'}$ and $Q_{r'}'=Q_{r'}$, then since $(P,Q)$ is assumed to be a negative Skew-symmetric notched bitableau, it follows immediately that $P_{r'}'-Q_{r'}'<\es$. Otherwise, we do the following: Let $P_{r'}'$ and $Q_{r'}'$ be given by $\lambda_1<\ldots<\lambda_{s'}$ and $\delta_1<\ldots<\delta_{s'}$ respectively. It follows easily from the way the OBRSK algorithm works that there exists at least one entry in $P_{r'}'$ which is $<b_t$. Let $\lambda_l$ denote the largest entry in $P_{r'}'$ which is $<b_t$. Since $b_t\leq\delta_1$, it now follows easily that $\lambda_j<\delta_j\ \forall j\in\{1,\ldots,l\}$. Since $\lambda_l$ is the largest entry in $P_{r'}'$ which is $<b_t$, therefore $l\geq 1$ and it follows from duality that $\delta_{s'+1-l}=\delta_{s'-(l-1)}>d_1$. Hence $\delta_{s'}\geq\delta_{s'-(l-1)}>d_1=\lambda_{s'}$. It only remains to show that if $l+1<s'$, then $\lambda_j<\delta_j\ \forall j\in\{l+1,\ldots,s'-1\}$. Clearly since $l\geq 1$ and $l+1<s'$, it follows that in this case, the last row of $P(resp.\ Q)$ is the $r'$-th row, namely $P_{r'}(resp.\ Q_{r'})$ and $\delta_{s'-(l-1)}$ is the new box of the dual insertion in $Q_{r'}$. Since $\lambda_{s'}=d_1, \delta_{s'-(l-1)}>d_1$, $P_{r'}-Q_{r'}<\es$ by induction hypothesis, and $\lambda_{j}<\lambda_{s'}\forall\ j\in\{l+1,\ldots,s'-1\}$, it is now easy to see that $\lambda_j<\delta_j\ \forall j\in\{l+1,\ldots,s'-1\}$. 

Hence we have proved that $(P',Q')$ is a negative semistandard notched bitableau. The fact that $(P',Q')$ is Skew-symmetric is easy to see from the very construction of the $OBRSK$ algorithm and from the fact that the pair $\{\pi_1^{(t)},\pi_2^{(t)}\}$ of arrays is Skew-symmetric, lexicographic. 
\end{Proof}
\subsection{Proof of Lemma \ref{l.g.dominance}}\label{ss.l.g.dominance}
\begin{Proof}
Let $\{U_1,U_2\}$ be a pair of non-vanishing skew-symmetric multisets on $\NN$, and let $T$ and $W$ be negative and positive subsets of $\NN$ respectively, with the property that condition \ref{e.g.Kreiman's condition (*)} is satisfied. Lemma \ref{l.g.dominance} is part(v) of the following lemma.
\begin{lem}
(i) Suppose that $\{U_1=\{(a_1,b_1),\ldots,(a_t,b_t)\},U_2=\{(d_1,c_1),\ldots,(d_t,c_t)\}\}$ is a pair of negative skew-symmetric 
multisets on $\NN$ such that ${\psi}^{-1}(\{U_1,U_2\})=\{\pi_1,\pi_2\}$ where
\begin{equation}\pi_1= \left(\begin{array}{ccc}
b_1 & \cdots & b_{t}\\
a_1 & \cdots & a_{t}\\ \end{array}\right)\ ,\ \pi_2= \left(\begin{array}{ccc}
c_1 & \cdots & c_{t}\\
d_1 & \cdots & d_{t}\\ \end{array}\right)\end{equation}
and $\{\pi_1,\pi_2\}$ is a pair of negative skew-symmetric lexicographic arrays.
For $k=1,\ldots,t$, let $U_1^{(k)}:=\{(a_1,b_1),\ldots,(a_k,b_k)\}$ and $U_2^{(k)}:=\{(d_{t+1-k},c_{t+1-k}),\ldots,(d_t,c_t)\}$.
Let $\{\pi_1^k,\pi_2^k\}:={\psi}^{-1}(\{U_1^{(k)},U_2^{(k)}\})$. Let $(P^{(k)},Q^{(k)})=OBRSK(\{\pi_1^k,\pi_2^k\})$ (note that
$(P^{(t)},Q^{(t)})=OBRSK(\{\pi_1^t,\pi_2^t\})$). Define
$\{p^{(k)}_1,\ldots,p^{(k)}_{2c_k}\}$ to be the topmost row of
$P^{(k)}$ and $\{q^{(k)}_1,\ldots,q^{(k)}_{2c_k}\}$ the topmost row
of $Q^{(k)}$, both listed in increasing order. Let
$m(k):=\max\{m\in\{1,\ldots,2c_k\}\mid
p^{(k)}_{m}<q^{(k)}_1\}=|(P^{(k)}_1)^{<q^{(k)}_1}|$. Then for
$1\leq j\leq m(k)$, there exists a dual pair of chains $\{C_{k,j}^{(1)},C_{k,j}^{(2)}\}$ in $\{U_1^{(k)},U_2^{(k)}\}$ 
which has at most $j$ elements each in $C_{k,j}^{(1)}$ and $C_{k,j}^{(2)}$, and there exists an integer $\chi_{k,j}$ which is $\geq j$ such that the first coordinate of the $\chi_{k,j}$-th element of $(P_{C_{k,j}^{(1)-},C_{k,j}^{(2)-}},Q_{C_{k,j}^{(1)-},C_{k,j}^{(2)-}})^{up}$ is $p^{(k)}_j$ and, all the entries which occur as first coordinates of elements of $C_{k,j}^{(1)}$ form a subset of the set of all entries which occur as first coordinates of elements of $(P_{C_{k,j}^{(1)-},C_{k,j}^{(2)-}},Q_{C_{k,j}^{(1)-},C_{k,j}^{(2)-}})^{up}_{\leq\chi_{k,j}}$.\\
\noindent (ii) If $\{U_1,U_2\}$ is bounded by $T,\es$, then
$(P^{(k)},Q^{(k)})$ is bounded by $T,\es$ for all $k=1,\ldots,t$.\\
\noindent (iii) If $\{U_1,U_2\}$ is
bounded by $T,\es$, then $OBRSK(\{U_1,U_2\})$ is bounded by $T,\es$.\\
\noindent (iv) If $\{U_1,U_2\}$ is bounded by $\es,W$, then $OBRSK(\{U_1,U_2\})$ is
bounded by $\es,W$.\\
\noindent (v) If $\{U_1,U_2\}$ is bounded by $T,W$, then $OBRSK(\{U_1,U_2\})$ is
bounded by $T,W$.
\end{lem}
\begin{proof}
We prove (i) and (ii) together by induction on $k$, with $k=1$ the starting point of induction. When $k=1$, $U_1^{(1)}=\{(a_1,b_1)\}$ and $U_2^{(1)}=\{(d_t,c_t)\}$. Clearly then, $P^{(1)}$ is given by a single row tableau containing the two elements $a_1$ and $d_t$ where $a_1<d_t$ and, $Q^{(1)}$ is given by a single row tableau containing the two elements $b_1$ and $c_t$ where $b_1<c_t$. 

For (i), there are two possible cases, namely when $m(k)=1$ and $m(k)=2$. In both the cases, for every $j\in\{1,\ldots,m(k)\}$, take $C_{1,j}^{(1)}=\{(a_1,b_1)\}$ and $C_{1,j}^{(2)}=\{(d_t,c_t)\}$.

For (ii), $\{U_1,U_2\}$ is bounded by $T,\es$ implies that for the dual pair $\{C_{1,j}^{(1)},C_{1,j}^{(2)}\}$ of chains in $\{U_1,U_2\}$ as mentioned above, we have:---
\begin{equation}\label{e.g.1_is_bounded_by_T,phi}
T\leq (P_{C_{1,j}^{(1)-},C_{1,j}^{(2)-}},Q_{C_{1,j}^{(1)-},C_{1,j}^{(2)-}})^{up}\leq \es
\end{equation}
But it can be easily seen that $(P_{C_{1,j}^{(1)-},C_{1,j}^{(2)-}},Q_{C_{1,j}^{(1)-},C_{1,j}^{(2)-}})^{up}=\{(a_1,b_1),(d_t,c_t)\}$. So equation \ref{e.g.1_is_bounded_by_T,phi} above is clearly equivalent to saying that $(P^{(1)},Q^{(1)})$ is bounded by $T,\es$. This proves parts (i) and (ii) of the base case of induction.

Now let $k\in 1,\ldots,t-1$.  Let $(P,Q)=(P^{(k)},Q^{(k)})$,
$a=a_{k+1}$, $b=b_{k+1}$, $c=c_{t-k}$, $d=d_{t-k}$, $(P',Q')=(P^{(k+1)},Q^{(k+1)})$,
$\{V_1,V_2\}=\{U_1^{(k)},U_2^{(k)}\}$, $\{V_1',V_2'\}=\{U_1^{(k+1)},U_2^{(k+1)}\}$,
$\{p_1,\ldots,p_{2\hat{c}}\}=\{p^{(k)}_1,\ldots,p^{(k)}_{2c_k}\}$ and,
$\{q_1,\ldots,q_{2\hat{c}}\}=\{q^{(k)}_1,\ldots,q^{(k)}_{2c_k}\}$. Note that
$\{p_1,\ldots,p_{2\hat{c}}\}\subset \{a_1,\ldots,a_k\}\disjunion\{d_{t+1-k},\ldots,d_t\}$ and $\{q_1,\ldots,q_{2\hat{c}}\}\subset \{b_1,\ldots,b_k\}\disjunion\{c_{t+1-k},\ldots,c_t\}$. Let $P_1$ (resp.$Q_1$) denote the topmost row of $P$ (resp. $Q$). Similarly let $P_1'$ (resp. $Q_1'$) denote the topmost row of $P'$ (resp. $Q'$).

Since $b$ is
less than or equal to all elements of $\{b_1,\ldots,b_k\}$ and $b_i<c_{t+1-i}\ \forall\ i\in\{1,\ldots,t\}$,
it follows that $b\leq$ all elements of $\{b_1,\ldots,b_k\}\disjunion\{c_{t+1-k},\ldots,c_t\}$. Therefore $a<b\leq q_1$, and hence by duality $c>d\geq p_{2\hat{c}}$. We assume inductively that $$T_{(1)}-T_{(2)}\leq
P_1-Q_1,$$ and we need to prove that
$$T_{(1)}-T_{(2)}\leq P'_1-Q'_1.$$ Equivalently, we need to prove that
for all positive integers $z$, $$|(T_{(1)}-T_{(2)})^{\leq z}|
\geq|(P'_1-Q'_1)^{\leq z}|,$$ where we use the definition
$A-B:=A\disjunion (\bN\setminus B)$, where $A$ and $B$ are both
subsets of $\bN$ (see Section 4 of \cite{Kr-bkrs}). 

We consider two cases corresponding to the two ways in which $(P_1',Q_1')$ can be obtained from $(P_1,Q_1)$.

\noindent {\it Case 1}.  $P'_1$ is obtained by $a$ bumping $p_l$
in $P_1$, for some $1\leq l\leq 2\hat{c}$, i.e.,
\[
\begin{array}{l}
P'_1=(P_1\setminus \{p_l\})\disjunion \{a\}\\
Q'_1=(Q_1\setminus \{q_{2\hat{c}+1-l}\})\disjunion \{c\}
\end{array}
\]

\noindent (i) The fact that $a$ bumps $p_l$ implies that $a\leq p_l$ and $p_l<b$. Hence $a\leq p_l<b\leq q_1$ and therefore by duality, we have $c\geq q_{2\hat{c}+1-l}>d\geq p_{2\hat{c}}$. This implies that $m(k+1)=m(k)$. For $j\in\{1,\ldots,m(k)\}\setminus\{l\}$, set $C_{k+1,j}^{(1)}=C_{k,j}^{(1)}$ and $C_{k+1,j}^{(2)}=C_{k,j}^{(2)}$ (note that in these cases $p_j^{(k)}=p_j^{(k+1)}$.). We now consider the case when $j=l$. If $l=1$, then set $C_{k+1,l}^{(1)}=\{(a,b)\}$ and $C_{k+1,l}^{(2)}=\{(d,c)\}$. Otherwise consider the dual pair of chains $\{C_{k,l-1}^{(1)},C_{k,l-1}^{(2)}\}$ in $\{U_1^{(k)},U_2^{(k)}\}$. 

By induction hypothesis, there are at most $(l-1)$ elements each in $C_{k,l-1}^{(1)}$ and $C_{k,l-1}^{(2)}$, and there exists an integer $\chi_{k,l-1}(\geq l-1)$ such that the first coordinate of the $\chi_{k,l-1}$-th element of ${(P_{C_{k,l-1}^{(1)-},C_{k,l-1}^{(2)-}},Q_{C_{k,l-1}^{(1)-},C_{k,l-1}^{(2)-}})}^{up}$ is $p_{l-1}$ and, all the entries which occur as first coordinates of elements of $C_{k,l-1}^{(1)}$ form a subset of the set of all entries which occur as first coordinates of elements of ${(P_{C_{k,l-1}^{(1)-},C_{k,l-1}^{(2)-}},Q_{C_{k,l-1}^{(1)-},C_{k,l-1}^{(2)-}})}^{up}_{\leq\chi_{k,l-1}}$.

Say, $C_{k,l-1}^{(1)}=\{(e_1,f_1),\ldots,(e_r,f_r)\}$ and $C_{k,l-1}^{(2)}=\{(g_r,h_r),\ldots,(g_1,h_1)\}$ where $r\leq l-1$. Therefore $e_1<\cdots<e_r$ and $f_1>\dots>f_r$. It follows from the induction hypothesis $e_1<\cdots<e_r\leq p_{l-1}$. Since $a$ bumps $p_l$, it follows that $a>p_{l-1}$. Hence $e_1<\cdots<e_r<a$. Also $b<f_r$, because $(a,b)$ comes after $(e_r,f_r)$ in the ordered list of elements of $V_1'$.

Therefore $C_{k,l-1}^{(1)}\cup\{(a,b)\}$ is a chain in $V_1'$. We let $C_{k+1,l}^{(1)}$ to be this chain. It follows from duality that $C_{k,l-1}^{(2)}\cup\{(d,c)\}$ is a chain in $V_2'$. We let $C_{k+1,l}^{(2)}$ to be this chain. Note that the dual pair $\{C_{k+1,l}^{(1)},C_{k+1,l}^{(2)}\}$ of chains in $\{U_1^{(k+1)},U_2^{(k+1)}\}$ satisfies the required conditions.

\noindent (ii) For $z<a$ or $p_l\leq z<q_{2\hat{c}+1-l}$ or $z\geq c$,
\begin{equation}\label{e.chd1}
|(T_{(1)}-T_{(2)})^{\leq z}|\geq|(P_1-Q_1)^{\leq
z}|=|(P'_1-Q'_1)^{\leq z}|.
\end{equation}
where the first inequality follows from induction hypothesis, and second equality follows from the facts that $p_{l-1}<a\leq p_l<b\leq q_1\leq q_{2\hat{c}+1-l}\leq c$ and $p_{2\hat{c}}\leq d<c$.

If $a=p_l$, then we are done. Thus we assume that $a<p_l$ (and hence by duality that $c>q_{2\hat{c}+1-l}$). We now need to consider only two possible positions of $z$, namely: $a\leq z<p_l$ and $q_{2\hat{c}+1-l}\leq z<c$. We claim that for $z$ such that $a\leq z<p_l$ or $q_{2\hat{c}+1-l}\leq z<c$,
\begin{equation}\label{e.chd2}
|({(P_{C_{k+1,l}^{(1)-},C_{k+1,l}^{(2)-}},Q_{C_{k+1,l}^{(1)-},C_{k+1,l}^{(2)-}})}^{up}_{(1)}-{(P_{C_{k+1,l}^{(1)-},C_{k+1,l}^{(2)-}},Q_{C_{k+1,l}^{(1)-},C_{k+1,l}^{(2)-}})}^{up}_{(2)})^{\leq z}|\geq |(P'_1-Q'_1)^{\leq z}|.
\end{equation}
Assuming the claim and using the fact that $T\leq {(P_{C_{k+1,l}^{(1)-},C_{k+1,l}^{(2)-}},Q_{C_{k+1,l}^{(1)-},C_{k+1,l}^{(2)-}})}^{up}$ (which is because $\{U_1,U_2\}$ is bounded by $T,\es$), we have that for $z$ such that $a\leq z<p_l$ or $q_{2\hat{c}+1-l}\leq z<c$,
$$|(T_{(1)}-T_{(2)})^{\leq z}|\geq |({(P_{C_{k+1,l}^{(1)-},C_{k+1,l}^{(2)-}},Q_{C_{k+1,l}^{(1)-},C_{k+1,l}^{(2)-}})}^{up}_{(1)}-{(P_{C_{k+1,l}^{(1)-},C_{k+1,l}^{(2)-}},Q_{C_{k+1,l}^{(1)-},C_{k+1,l}^{(2)-}})}^{up}_{(2)})^{\leq z}|$$\ 
$$\geq |(P'_1-Q'_1)^{\leq z}|.$$
This proves the inductive step of (ii). We will now prove the claim.

Note that $C_{k+1,l}^{(1)-}=C_{k+1,l}^{(1)}$ and $C_{k+1,l}^{(2)-}=C_{k+1,l}^{(2)}$. From the proof of (i), we have that $C_{k+1,l}^{(1)}=\{(e_1,f_1),\ldots,(e_r,f_r),(a,b)\}$ and $C_{k+1,l}^{(2)}=\{(d,c),(g_r,h_r),\ldots,(g_1,h_1)\}$ where $e_1<\cdots<e_r<a<p_l<b<f_r<\cdots<f_1$ and $h_1>\cdots>h_r>c>q_{2\hat{c}+1-l}>d>g_r>\cdots>g_1$.

Thus for $a\leq z<p_l$, we have,
$$|({(P_{C_{k+1,l}^{(1)},C_{k+1,l}^{(2)}},Q_{C_{k+1,l}^{(1)},C_{k+1,l}^{(2)}})}^{up}_{(1)}-{(P_{C_{k+1,l}^{(1)},C_{k+1,l}^{(2)}},Q_{C_{k+1,l}^{(1)},C_{k+1,l}^{(2)}})}^{up}_{(2)})^{\leq z}|$$
$$=|(\{topmost\ row\ of\ P_{C_{k+1,l}^{(1)},C_{k+1,l}^{(2)}}\}-\{topmost\ row\ of\ Q_{C_{k+1,l}^{(1)},C_{k+1,l}^{(2)}}\})^{\leq z}|$$
$$=|(\{topmost\ row\ of\ P_{C_{k+1,l}^{(1)},C_{k+1,l}^{(2)}}\}\disjunion(\bN\setminus\{topmost\ row\ of\ Q_{C_{k+1,l}^{(1)},C_{k+1,l}^{(2)}}\}))^{\leq z}|$$
$$=|\{topmost\ row\ of\ P_{C_{k+1,l}^{(1)},C_{k+1,l}^{(2)}}\}^{\leq z}\disjunion\{\bN\}^{\leq z}|\geq \chi_{k+1,l}+z\geq l+z$$

where the last equality (not inequality!) is because $b\leq$ all entries in $Q_{C_{k+1,l}^{(1)},C_{k+1,l}^{(2)}}$ and $a\leq z<p_l<b$ (All the other inequalities and equalities being obvious.). 

Also, $p_1<\cdots<p_{l-1}<a<p_l<b\leq q_1<\cdots<q_{2\hat{c}}$ and $b<c$. Thus for $a\leq z<p_l$,
$$|(P'_1-Q'_1)^{\leq z}|=|(P_1'\disjunion(\bN\setminus Q_1'))^{\leq z}|=|(P_1'\disjunion\bN)^{\leq z}|=l+z$$

Hence we have proved the claim for the case $a\leq z<p_l$. Now for $z$ such that $q_{2\hat{c}+1-l}\leq z<c$, we need to prove the claim, i.e., we need to show that
$$|(\{topmost\ row\ of\ P_{C_{k+1,l}^{(1)},C_{k+1,l}^{(2)}}\}\disjunion(\bN\setminus\{topmost\ row\ of\ Q_{C_{k+1,l}^{(1)},C_{k+1,l}^{(2)}}\}))^{\leq z}|$$
$$\geq |(P_1'\disjunion(\bN\setminus Q_1'))^{\leq z}|$$

Recall that $P_1'=(P_1\setminus\{p_l\})\disjunion\{a\}$ and $Q_1'=(Q_1\setminus\{q_{2\hat{c}+1-l}\})\disjunion\{c\}$. Since $(P',Q')=(P,Q)\la{b,c}a,d$, therefore $d\geq$ all entries of $P'$. Hence $d\geq$ all entries of $P_1'$. On the other hand, since $a<p_l<b$, it follows from duality that $c>q_{2\hat{c}+1-l}>d$. So we have $p_1<\cdots<p_{l-1}<a<p_{l+1}<\cdots<p_{2\hat{c}}\leq d<q_{2\hat{c}+1-l}<c<q_{2\hat{c}+1-(l-1)}<\cdots<q_{2\hat{c}}$. It is now easy to observe that for $z$ such that $q_{2\hat{c}+1-l}\leq z<c$, the number of elements in $Q_1'$ which are $\leq z$ is $2\hat{c}-l$. Hence the number of elements in $\bN\setminus Q_1'$ which are $\leq z$ will be $z-(2\hat{c}-l)$. 

On the other hand, it is also clear that all the elements of $P_1'$ are $\leq z$ and there are $2\hat{c}$ many elements in $P_1'$. Therefore,
$$|(P_1'\disjunion(\bN\setminus Q_1'))^{\leq z}|=2\hat{c}+(z-(2\hat{c}-l))=2\hat{c}+z-2\hat{c}+l=z+l$$

Let $\alpha_1<\cdots<\alpha_{2\tilde{c}}$ and $\beta_1<\cdots<\beta_{2\tilde{c}}$ denote the topmost rows of $P_{C_{k+1,l}^{(1)},C_{k+1,l}^{(2)}}$ and $Q_{C_{k+1,l}^{(1)},C_{k+1,l}^{(2)}}$ respectively. It follows from the definition of $\{C_{k+1,l}^{(1)}, C_{k+1,l}^{(2)}\}$ and from the algorithm of $OBRSK$ applied on the pair of arrays corresponding to $\{C_{k+1,l}^{(1)}, C_{k+1,l}^{(2)}\}$ that $d\geq\alpha_{2\hat{c}}>\cdots>\alpha_1$. On the other hand, since $a<p_l<b$, it follows from duality that $c>q_{2\hat{c}+1-l}>d$. Hence combining all these, we have $c>q_{2\hat{c}+1-l}>d\geq\alpha_{2\tilde{c}}>\cdot\alpha_1$. So for $z$ such that $q_{2\hat{c}+1-l}\leq z<c$, the number of elements in the topmost row of $P_{C_{k+1,l}^{(1)},C_{k+1,l}^{(2)}}$ which are $\leq z$ is $2\tilde{c}$. 

We know from (i) that there exists an integer $\chi_{k+1,l} (\geq l)$ such that the $\chi_{k+1,l}$-th entry (counting from left to right) of the topmost row of $P_{C_{k+1,l}^{(1)},C_{k+1,l}^{(2)}}$ is $p_l^{(k+1)}=a$. Hence it follows from duality that the backward $\chi_{k+1,l}$-th entry (i.e., the $\chi_{k+1,l}$-th entry counting from right to left) of the topmost row of $Q_{C_{k+1,l}^{(1)},C_{k+1,l}^{(2)}}$ is $c$. Therefore for $z$ such that $q_{2\hat{c}+1-l}\leq z<c$, the number of elements in the topmost row of $Q_{C_{k+1,l}^{(1)},C_{k+1,l}^{(2)}}$ which are $\leq z$ is equal to $X_0$ where $X_0$ is some non-negative integer such that $X_0\leq 2\tilde{c}-\chi_{k+1,l}$. But $\chi_{k+1,l}\geq l$, hence $-\chi_{k+1,l}\leq -l$ and therefore $X_0\leq 2\tilde{c}-\chi_{k+1,l}\leq 2\tilde{c}-l$.

Therefore, the number of elements in $(\bN\setminus\{topmost\ row\ of\ Q_{C_{k+1,l}^{(1)},C_{k+1,l}^{(2)}}\})$ which are $\leq z$ is $z-X_0$. Hence,
$$|(\{topmost\ row\ of\ P_{C_{k+1,l}^{(1)},C_{k+1,l}^{(2)}}\}\disjunion(\bN\setminus\{topmost\ row\ of\ Q_{C_{k+1,l}^{(1)},C_{k+1,l}^{(2)}}\}))^{\leq z}|$$
$$=2\tilde{c}+z-X_0\geq 2\tilde{c}+z-(2\tilde{c}-l)=z+l=|(P_1'\disjunion(\bN\setminus Q_1'))^{\leq z}|$$. This proves the claim in case 1.

\noindent {\it Case 2}. $P_1'$ is obtained by adding $a$ to $P_1$ in position $l$ from the left and adding $d$ to $P_1$ at the rightmost end (after $p_{2\hat{c}}$), $Q_1'$ is obtained from $Q_1$ by adding $b$ to the leftmost end of $Q_1$ and adding $c$ at the backward $l$-th position of $Q_1$. That is, 
$$P_1'=P_1\disjunion\{a,d\}=\{p_1,\ldots,p_{l-1},a,p_l,\ldots,p_{2\hat{c}},d\}\ and$$
$$Q_1'=Q_1\disjunion\{b,c\}=\{b,q_1,\ldots,q_{2\hat{c}+1-l},c,q_{2\hat{c}+1-(l-1)},\ldots,q_{2\hat{c}}\}$$
where $p_1<\cdots<p_{l-1}<a<p_l<\cdots<p_{2\hat{c}}<d$ and $b<q_1<\cdots<q_{2\hat{c}+1-l}<c<q_{2\hat{c}+1-(l-1)}<\cdots<q_{2\hat{c}}$.

(i) Since $p_{l-1}<a<b<q_1$, it follows that $m(k)\geq (l-1)$. Note that $a<b\leq p_l$ (since $b>p_l$ would require that $a$ bump $p_l$ in the bounded insertion process.). Thus $m(k+1)=l$.

For $j\in\{1,\ldots,l-1\}$, set $C_{k+1,j}^{(1)}=C_{k,j}^{(1)}$ and $C_{k+1,j}^{(2)}=C_{k,j}^{(2)}$. Consider the dual pair  $\{C_{k,l-1}^{(1)},C_{k,l-1}^{(2)}\}$ of chains in $\{U_1^{(k)},U_2^{(k)}\}$. Say, $C_{k,l-1}^{(1)}=\{(e_1,f_1),\ldots,(e_r,f_r)\}$ and $C_{k,l-1}^{(2)}=\{(g_r,h_r),\ldots,(g_1,h_1)\}$ where $r\leq l-1$. Therefore $e_1<\cdots<e_r$ and $f_1>\cdots>f_r$. It follows from the induction hypothesis that $e_1<\cdots<e_r\leq p_{l-1}$. Since $p_{l-1}<a$, we have $e_1<\cdots<e_r<a$. Also $b<f_r$ because $(a,b)$ comes after $(e_r,f_r)$ in the ordered list of elements of $V_1'$. Therefore $C_{k,l-1}^{(1)}\cup\{(a,b)\}$ is a chain in $V_1'$. We let $C_{k+1,l}^{(1)}$ to be this chain. It follows from duality that $C_{k,l-1}^{(2)}\cup\{(d,c)\}$ is a chain in $V_2'$. We let $C_{k+1,l}^{(2)}$ to be this chain. Note that the dual pair $\{C_{k+1,l}^{(1)},C_{k+1,l}^{(2)}\}$ of chains in $\{U_1^{(k+1)},U_2^{(k+1)}\}$ satisfies the required conditions. 

(ii) Note that $a<b<d<c$. For $z<a$,
$$|(T_{(1)}-T_{(2)})^{\leq z}|\geq|(P_1-Q_1)^{\leq
z}|=|(P'_1-Q'_1)^{\leq z}|$$
where the first inequality follows from induction hypothesis and the second equality follows from the facts that $p_{l-1}<a<p_l$ and $a<b<q_1$. 
For $z$ such that $b\leq z<d$, note that
\begin{align*}
|(P'_1-Q'_1)^{\leq z}|&=|(P'_1\disjunion (\bN\setminus
Q'_1))^{\leq z}|\\
&=|(P'_1)^{\leq z}|+|(\bN\setminus Q'_1)^{\leq
z}|\\
&=(|(P_1)^{\leq z}|+1)+(|(\bN\setminus Q_1)^{\leq
z}|-1)\\
&=|(P_1)^{\leq z}|+|(\bN\setminus Q_1)^{\leq z}|\\
&=|(P_1-Q_1)^{\leq z}|.
\end{align*}
Hence for $b\leq z<d$, we have $$|(T_{(1)}-T_{(2)})^{\leq z}|\geq|(P_1-Q_1)^{\leq
z}|=|(P'_1-Q'_1)^{\leq z}|$$.
For $z\geq c$,
\begin{align*}
|(P'_1-Q'_1)^{\leq z}|&=|(P'_1\disjunion (\bN\setminus
Q'_1))^{\leq z}|\\
&=|(P'_1)^{\leq z}|+|(\bN\setminus Q'_1)^{\leq
z}|\\
&=(|(P_1)^{\leq z}|+2)+(|(\bN\setminus Q_1)^{\leq
z}|-2)\\
&=|(P_1)^{\leq z}|+|(\bN\setminus Q_1)^{\leq z}|\\
&=|(P_1-Q_1)^{\leq z}|.
\end{align*}
Hence for $z\geq c$, we have $$|(T_{(1)}-T_{(2)})^{\leq z}|\geq|(P_1-Q_1)^{\leq
z}|=|(P'_1-Q'_1)^{\leq z}|.$$
It now remains to show that for $z$ such that $a\leq z<b$ or $d\leq z<c$,
$$|(T_{(1)}-T_{(2)})^{\leq z}|\geq |(P'_1-Q'_1)^{\leq z}|.$$
We claim that for $z$ such that $a\leq z<b$ or $d\leq z<c$,
$$|(\{topmost\ row\ of\ P_{C_{k+1,l}^{(1)},C_{k+1,l}^{(2)}}\}\disjunion(\bN\setminus\{topmost\ row\ of\ Q_{C_{k+1,l}^{(1)},C_{k+1,l}^{(2)}}\}))^{\leq z}|$$
$$\geq |(P_1'-Q_1')^{\leq z}|$$
Assuming the claim, we are done as in Case 1. We now prove the claim. Let us first consider the case when $a\leq z<b$. Then since $b\leq$ all entries in $Q_{C_{k+1,l}^{(1)},C_{k+1,l}^{(2)}}$, it follows that 
$$|(\{topmost\ row\ of\ P_{C_{k+1,l}^{(1)},C_{k+1,l}^{(2)}}\}\disjunion(\bN\setminus\{topmost\ row\ of\ Q_{C_{k+1,l}^{(1)},C_{k+1,l}^{(2)}}\}))^{\leq z}|$$
$$=|\{topmost\ row\ of\ P_{C_{k+1,l}^{(1)},C_{k+1,l}^{(2)}}\}^{\leq z}\disjunion\{\bN\}^{\leq z}|$$
which in turn is $\geq\chi_{k+1,l}+z\geq l+z$.
Also $p_1<\cdots<p_{l-1}<a<b<q_1$ and $b\leq p_l$. Thus for $z$ such that $a\leq z<b$, 
$$|(P_1'-Q_1')^{\leq z}|=|(P_1'\disjunion(\bN\setminus Q_1'))^{\leq z}|=|(P_1'\disjunion\bN)^{\leq z}|=l+z$$
This proves the claim in the case when $a\leq z<b$.

Now let us consider the case when $d\leq z<c$. Since $d\geq$ all entries of $P_1'$, it follows that $|(P_1')^{\leq z}|=2\hat{c}+2$. On the other hand, since $b\leq p_l$, it follows from duality that $d\geq q_{2\hat{c}+1-l}$. Therefore, the number of elements in $Q_1'$ which are $\leq z$ is $2\hat{c}+1-l+1=2\hat{c}+2-l$. Hence the number of elements in $\bN\setminus Q_1'$ which are $\leq z$ will be $z-(2\hat{c}+2-l)$. Therefore,
$$|(P_1'-Q_1')^{\leq z}|=|(P_1'\disjunion(\bN\setminus Q_1'))^{\leq z}|=(2\hat{c}+2)+z-(2\hat{c}+2-l)=z+l$$.

Let $\alpha_1<\cdots<\alpha_{2\tilde{c}}$ and $\beta_1<\cdots<\beta_{2\tilde{c}}$ denote the topmost rows of $P_{C_{k+1,l}^{(1)},C_{k+1,l}^{(2)}}$ and $Q_{C_{k+1,l}^{(1)},C_{k+1,l}^{(2)}}$ respectively. It follows from the definition of $\{C_{k+1,l}^{(1)},C_{k+1,l}^{(2)}\}$ and from the algorithm of $OBRSK$ applied on the pair of arrays corresponding to $\{C_{k+1,l}^{(1)},C_{k+1,l}^{(2)}\}$ that $d\geq\alpha_{2\tilde{c}}>\cdots>\alpha_1$. Hence for $z$ such that $d\leq z<c$, the number of elements in the topmost row of $P_{C_{k+1,l}^{(1)},C_{k+1,l}^{(2)}}$ which are $\leq z$ is $2\tilde{c}$. 

We know from (i) that there exists an integer $\chi_{k+1,l}(\geq l)$ such that the $\chi_{k+1,l}$-th entry (counting from left to right) of the topmost row of $P_{C_{k+1,l}^{(1)},C_{k+1,l}^{(2)}}$ is $p_l^{(k+1)}=a$. Hence it follows from duality that the backward $\chi_{k+1,l}$-th entry (i.e.,the $\chi_{k+1,l}$-th entry counting from right to left) of the topmost row of $Q_{C_{k+1,l}^{(1)},C_{k+1,l}^{(2)}}$ is $c$.

Therefore, for $z$ such that $d\leq z<c$, the number of elements in the topmost row of $Q_{C_{k+1,l}^{(1)},C_{k+1,l}^{(2)}}$ which are $\leq z$ is equal to $X_0$ where $X_0$ is some non-negative integer such that $X_0\leq 2\tilde{c}-\chi_{k+1,l}$. But $\chi_{k+1,l}\geq l$, and therefore $X_0\leq 2\tilde{c}-\chi_{k+1,l}\leq 2\tilde{c}-l$. Therefore, the number of elements in $(\bN\setminus\{topmost\ row\ of\ Q_{C_{k+1,l}^{(1)},C_{k+1,l}^{(2)}}\})$ which are $\leq z$ is $z-X_0$. Hence, 
$$|(\{topmost\ row\ of\ P_{C_{k+1,l}^{(1)},C_{k+1,l}^{(2)}}\}\disjunion(\bN\setminus\{topmost\ row\ of\ Q_{C_{k+1,l}^{(1)},C_{k+1,l}^{(2)}}\}))^{\leq z}|$$
$$=2\tilde{c}+z-X_0\geq 2\tilde{c}+z-(2\tilde{c}-l)=z+l=|(P_1'-Q_1')^{\leq z}|$$. 
This proves the claim in Case 2. 

So, we are done with the proofs of (i) and (ii) in all possible cases.

\noindent (iii) Set $k=t$ in (ii).

\noindent (iv) Use arguments similar to (i), (ii), and (iii), but
for $\{U_1,U_2\}$ a pair of positive skew-symmetric multisets on $\NN$. Alternatively, one could apply the involution
$L$ to (iii).

\noindent (v) Use (iii), (iv), and the fact that $\{U_1,U_2\}$ is bounded by
$T,W$ if and only if $\{U_1^-,U_2^-\}$ is bounded by $T,\es$ and $\{U_1^+,U_2^+\}$ is
bounded by $\es,W$; and similarly for $OBRSK(\{U_1,U_2\})$.  
\end{proof}
\end{Proof}

\providecommand{\bysame}{\leavevmode\hbox
to3em{\hrulefill}\thinspace}
\providecommand{\MR}{\relax\ifhmode\unskip\space\fi MR }
\providecommand{\MRhref}[2]{%
  \href{http://www.ams.org/mathscinet-getitem?mr=#1}{#2}
} \providecommand{\href}[2]{#2}

\vspace{1em}

\noindent \textsc{School of Mathematics, Tata Institute of Fundamental Research,
Mumbai, INDIA 400005}

\noindent \textsl{Email address}: \texttt{shyama@math.tifr.res.in}

\end{document}